\documentclass[12pt]{amsart}

\usepackage{amsfonts, amstext, amsmath, amsthm, amscd, amssymb}
\usepackage{epsfig, psfrag, color, graphicx} 

\let\oldmarginpar\marginpar
\renewcommand\marginpar[1]{\oldmarginpar[\raggedleft\footnotesize #1]%
{\raggedright\footnotesize #1}}

\theoremstyle{plain}
\newtheorem{thm}{Theorem}[section]

\newtheorem{definition}[thm]{Definition}
\newtheorem{lem}[thm]{Lemma}
\newtheorem{lemma}[thm]{Lemma}
\newtheorem{prop}[thm]{Proposition}

\theoremstyle{definition}

\newtheorem{remark}[thm]{Remark}

\newcommand{\Q}{\mathbb{Q}}
\newcommand{\Z}{\mathbb{Z}}
\newcommand{\G}{\mathbb{G}}
\newcommand{\E}{||E_c||}
\newcommand{\X}{\chi_-}
\newcommand{\cutalong}{\backslash \backslash}
\newcommand{\QQ}{\mathbb{Q}}

\newcommand{\num}{{\rm num}}
\newcommand{\denom}{{\rm denom}}
\newcommand{\vol}{{\rm vol}}
\newcommand{\guts}{{\rm guts}}

\AtBeginDocument{%
   \def\MR#1{}
}

\begin{document}

\title{Volumes of Montesinos links}
\author{Kathleen Finlinson}
\address{Department of Applied Mathematics, University of Colorado, Boulder, Colorado, USA}
\author{Jessica S.~Purcell}
\address{Department of Mathematics, Brigham Young University, Provo, Utah, USA, and 
School of Mathematical Sciences, Monash University, Clayton, VIC, Australia}

\begin{abstract} 

We show that the volume of any Montesinos link can be bounded above and below in terms of the combinatorics of its diagram. This was known for Montesinos links with at most two tangles, and those with at least five tangles. We complete the result for the remaining cases. 

\end{abstract}

\maketitle

\section{Introduction}\label{sec:intro}
W.~Thurston proved that if $K$ is a non-torus, non-satellite knot, then its complement $S^3 \setminus K$ admits a complete hyperbolic metric \cite{thurston1982}. This metric is unique up to isometry by the Mostow--Prasad rigidity theorem~\cite{Mostow, Prasad}. Therefore, the hyperbolic volume of a knot complement is a knot invariant. This paper studies the hyperbolic volume of Montesinos links.

Montesinos links are built out of rational tangles, whose definition and properties we review below. It is known that any Montesinos link made up of just one or two rational tangles is a 2--bridge link, and therefore it admits an alternating diagram. Volumes of alternating links can be bounded below due to work of Lackenby \cite{lackenby}. On the other hand, Futer, Kalfagianni, and Purcell found lower volume bounds for Montesinos links with at least three positive tangles (\cite[Theorem~8.6 and~9.1]{FKP}); by taking the mirror image, this also gives lower bounds on Montesinos links with at least three negative tangles. Together, this gives lower volume bounds on all Montesinos links with five or more tangles. However, until now, volume bounds for Montesinos links with three or four tangles were unknown.

In this paper, we finish the case of Montesinos links with three or four tangles. We give a lower bound on the volume of any such link in terms of properties of a diagram, which can easily be read off the diagram of the Montesinos link. Specifically, we show volume is bounded in terms of the Euler characteristic of a graph obtained from the diagram. This graph is the reduced $A$ or $B$--state graph $\G_\sigma'$, defined in Definition~\ref{def:G_A} in Section~\ref{sec:estimatingtheguts}. Our main result is the following.

\begin{thm} 
\label{thm:KJmain}
Let $K$ be a hyperbolic Montesinos link with a reduced, admissible diagram with at least three tangles. Then
\[
\vol(S^3 \setminus K) \geq v_8 ( \X ( \G'_\sigma) - 1).
\]
Here $v_8 \approx 3.6638$ is the hyperbolic volume of a regular ideal octahedron, and $\G'_\sigma$ is the reduced state graph of $D(K)$ corresponding to either the all--$A$ or all--$B$ state, depending on whether the diagram of $K$ admits two or more positive tangles, or two or more negative tangles, respectively.
\end{thm}

The definitions of reduced and admissible diagrams are given in Section~\ref{sec:montesinoslinks}, Definitions~\ref{def:reduced} and~\ref{def:admissibleMont}, respectively. Every Montesinos link admits such a diagram.
The notation $\X(\cdot)$ denotes the negative Euler characteristic, defined to be
\[ \X(Y) = \sum \max\{-\chi(Y_i), 0\},\]
where the sum is over the components $Y_1, \dots, Y_n$ of $Y$. 

While Theorem~\ref{thm:KJmain} gives explicit diagrammatical bounds on volume, in many cases, we may estimate $\X(\G'_\sigma)$ in terms of the \emph{twist number} $t(K)$ of the diagram, which is even easier to read off of the diagram. The following theorem generalizes \cite[Theorem~9.12]{FKP}. 

\begin{thm}
\label{thm:KJ2}
Let $K$ be a Montesinos link that admits a reduced, admissible diagram with at least two positive tangles and at least two negative tangles, and suppose further that $K$ is not the $(2, -2, 2, -2)$ pretzel link. Then $K$ is hyperbolic, and 
\[
\frac{v_8}{4}(t(K)-\# K - 8) \leq \vol(S^3 \setminus K) \leq 2\,v_8\,t(K)
\]
where again $v_8 \approx 3.6638$ is the hyperbolic volume of a regular ideal octahedron, $t(K)$ is the twist number of the diagram, and $\#K$ is the number of link components of $K$.
\end{thm}

\subsection{Outline of Proof}
\label{sec:OutlineofProof}

Theorem~\ref{thm:KJmain} is proved by applying results in \cite{FKP}, but we will restate the relevant results in this paper to keep this paper self-contained. 
In \cite{FKP}, using the guts machinery of Agol, Storm, and Thurston \cite{AST}, it is shown that volumes of many links, including hyperbolic Montesinos links with at least two positive or two negative tangles, 
can be bounded below by identifying \emph{complex essential product disks (EPDs)} in the link complement, as defined in Definition~\ref{def:simplesemicomplex}. In particular, \cite[Theorem~9.3]{FKP} states that for diagrams of links satisfying particular hypotheses, which  include the Montesinos links of this paper, we have the following estimate:
\begin{equation}\label{eqn:FKPVol}
\vol(S^3\setminus K) \geq v_8\, (\X(\G'_A - ||E_c||),
\end{equation}
where $||E_c||$ is the number of complex essential product disks. 

In this paper, we show that for a Montesinos link with three or four tangles, the existence of a complex EPD leads to restrictions on the diagram. These restrictions, in turn, imply that at most one complex EPD may exist. This implies Theorem~\ref{thm:KJmain}. 

\subsection{Organization}
This paper is organized as follows. In Section~\ref{sec:montesinoslinks}, we review the definitions of rational tangles and Montesinos links. We will need to work with particular diagrams of these links, and we prove such diagrams exist and are prime. In Section~\ref{sec:estimatingtheguts}, we recall definitions of $A$--adequacy, and techniques from \cite{FKP} that can be applied to $A$--adequate links to give a polyhedron whose combinatorial description is determined by the diagram. We review these results and apply them to the Montesinos links of interest. Section~\ref{sec:epds} contains the main technical results in the paper. Given a polyhedron for a Montesinos link, we search for complex EPDs that lie in the polyhedron. These are found by analyzing the combinatorics of the diagram, and working through several cases. Finally, in Section~\ref{sec:proofs}, we put the results together to give the proofs of Theorems~\ref{thm:KJmain} and~\ref{thm:KJ2}. 

\subsection{Acknowledgments}
We aknowledge support by the National Science Foundation under grant number DMS--125687. We also thank David Futer and Efstratia Kalfagianni for helpful conversations.

\section{Tangles and Montesinos links}\label{sec:montesinoslinks}
In this section, we recall the definitions of rational tangles and Montesinos links, and various properties of their diagrams that we will use in the sequel. Throughout, if $K$ is a link in $S^3$, then $D(K)=D$ is the corresponding link diagram in the plane of projection, and we will assume that $D$ is connected. 

\subsection{Rational Tangles}
\label{sec:rationaltangles}

A \emph{rational tangles} is obtained by drawing two arcs of rational slope on the surface of a pillowcase, and then pushing the interiors into the $3$--ball bounded by the pillowcase. Rational tangles have been studied in many contexts, for example see \cite{murasugi}. We record here some well--known facts.

A rational number can be described by a continued fraction:
\[
\frac{p}{q} = [a_n, a_{n-1}, \hdots, a_1] =  a_n + \cfrac{1}{a_{n-1} + \cfrac{1}{\ddots \, + \cfrac{1}{a_1}}}.
\]

\begin{figure} 
\centering
\begin{tabular}{cccc}
\includegraphics[height=1in]{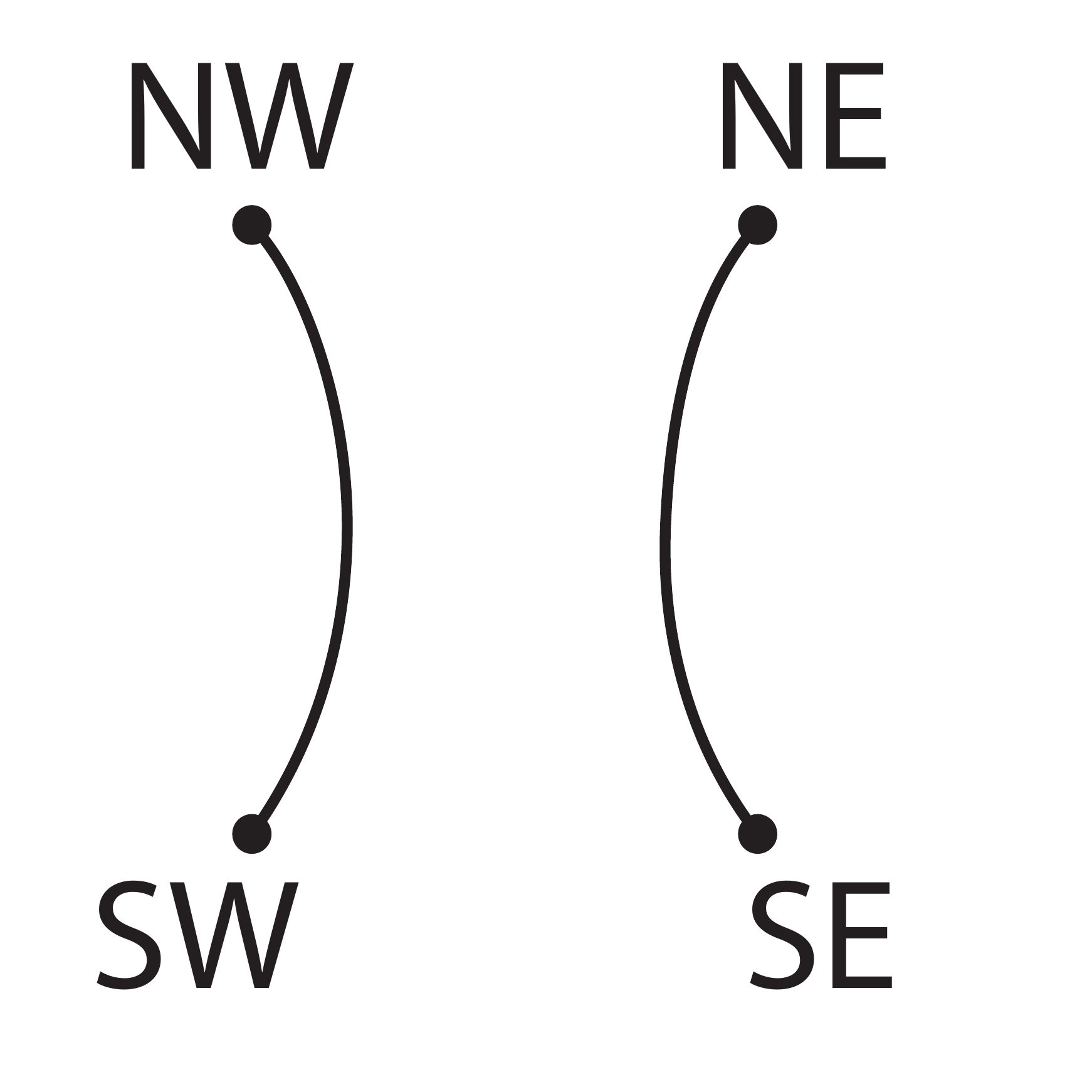} &
\includegraphics[height=1in]{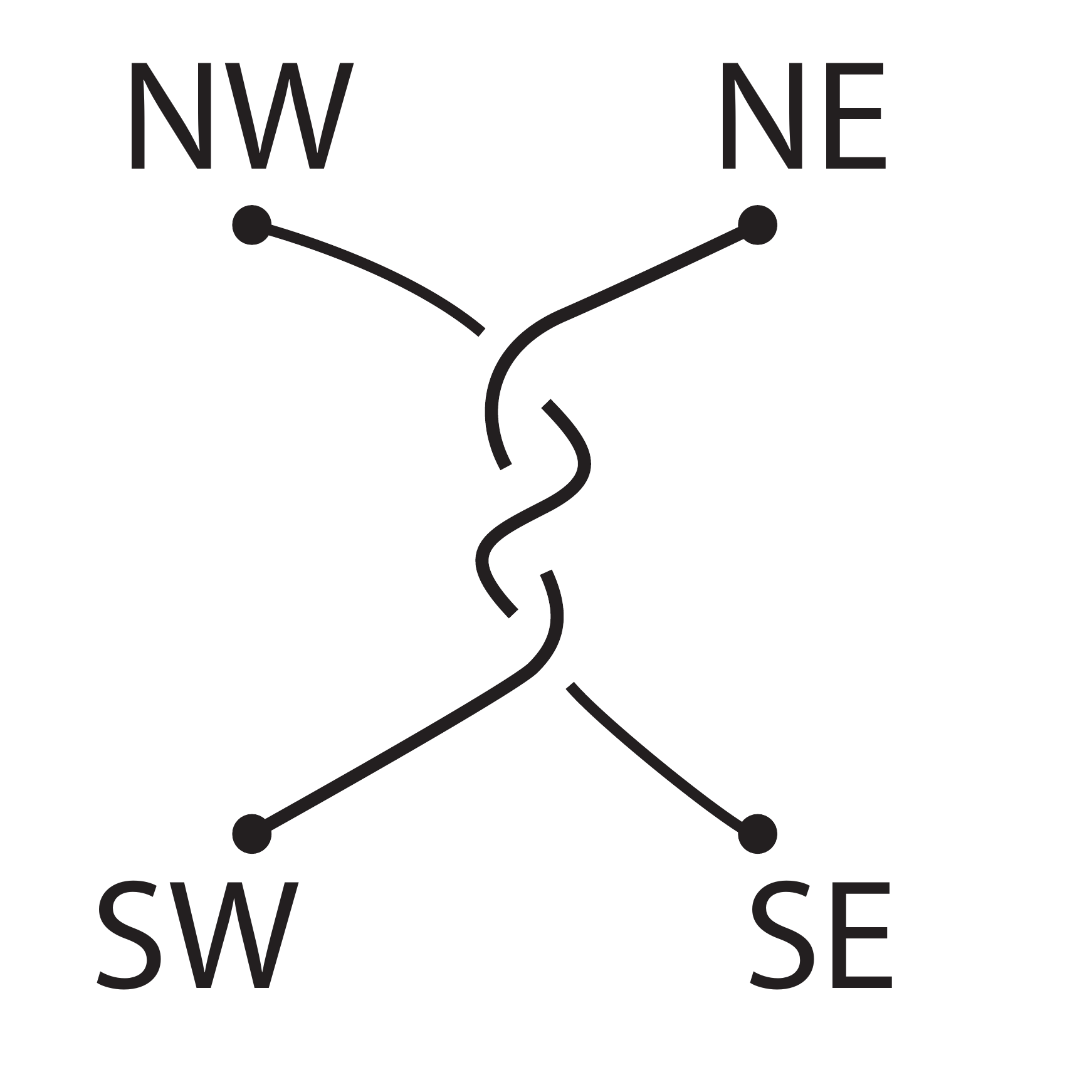} &
\includegraphics[height=1in]{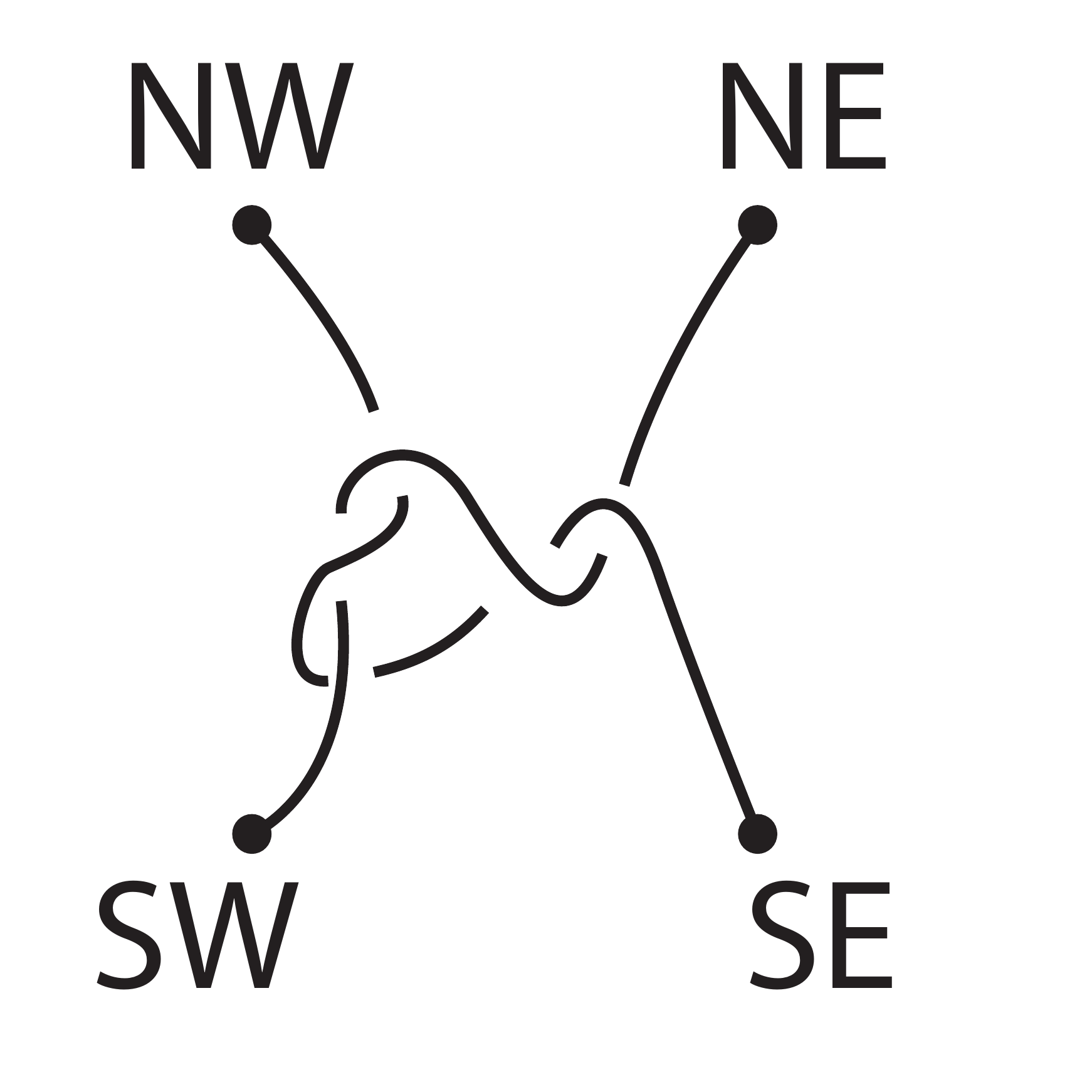} &
\includegraphics[height=1in]{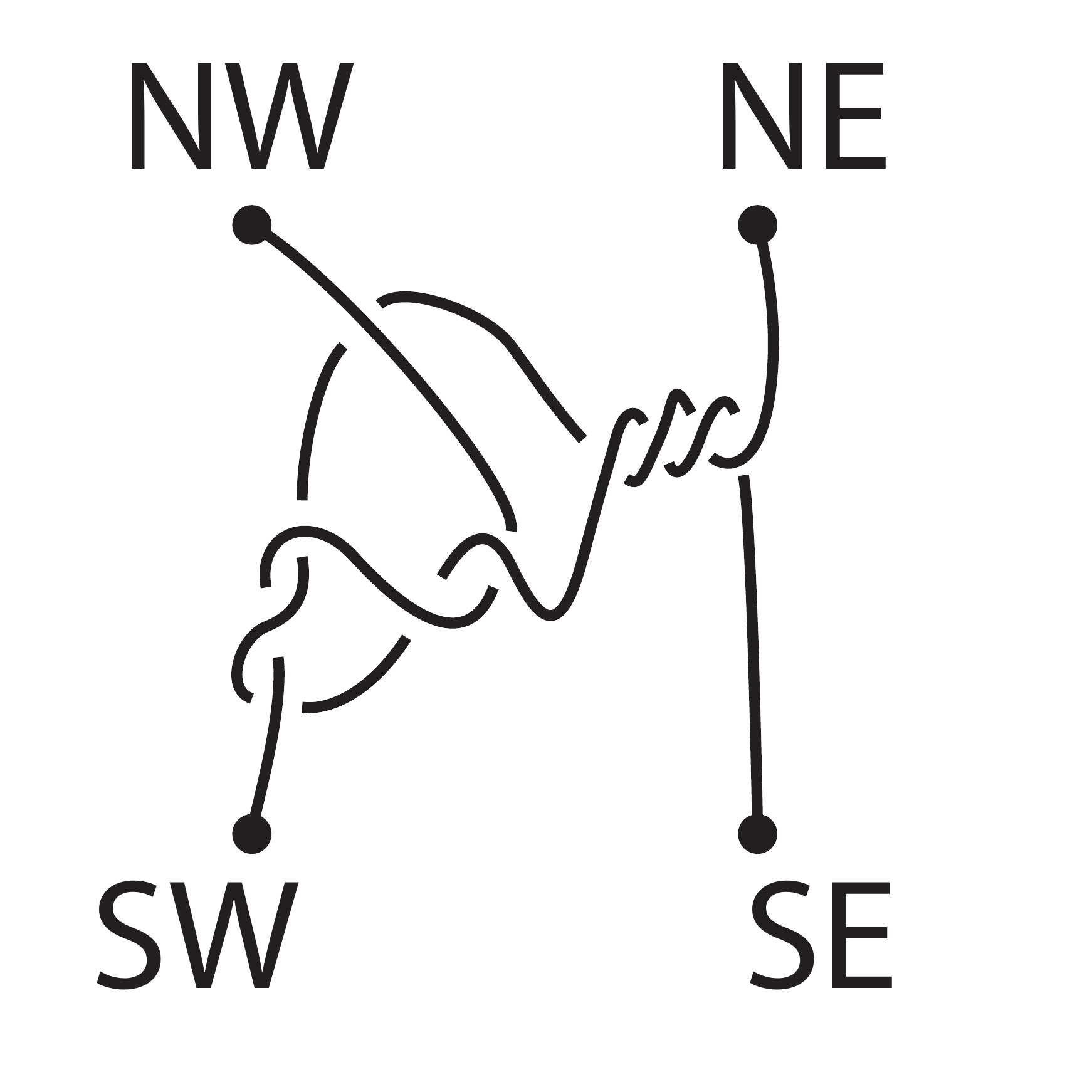} \\
(a) & (b) & (c) & (d)
\end{tabular}
\caption{Building a rational tangle from the continued fraction $[4, -1, -2, 3]$}
 \label{fig:RatTangle1}
\end{figure}

A continued fraction $[a_n, \dots, a_1]$ defines a rational tangle as follows. 
Label the four points on the pillowcase NW, NE, SW, and SE.
If $n$ is even, connect these points by attaching two arcs $c_1$ and $c_2$ connecting NE to SE and NW to SW as in Figure~\ref{fig:RatTangle1}(a). Perform a homeomorphism of $B^3$ that rotates the points NW and NE $|a_1|$ times, twisting the two arcs to create a vertical band of crossings. The crossings will be postive or negative depending on the direction of twist, which is determined by the sign of $a_1$. In Figure~\ref{fig:RatTangle1}(b), three positive crossings have been added. After twisting, relabel the points NW, NE, SW, and SE in their original orientation. Now perform a homeomorphism of $B^3$ to rotate NE and SE $|a_2|$ times, adding positive or negative crossings in a horizontal band with sign corresponding to $a_2$. Repeating this process for each $a_i$, we obtain a rational tangle.

If $n$ is odd, start by using two arcs to connect NW to NE and SW to SE. In this case we add a horizontal band of crossings first, and then continue as before, alternating between horizontal and vertical bands for each $a_i$.

Any rational tangle may be built by this process. As a convention, we require that $a_n$ always corresponds to a horizontal band of crossings. Thus if we build a rational tangle ending with a vertical band, as in Figure~\ref{fig:RatTangle1}(b), we insert a $0$ into the corresponding continued fraction, representing a horizontal band of $0$ crossings. For example, the continued fraction corresponding to the tangle in Figure \ref{fig:RatTangle1}(b) is $[0, 3]$. This convention ensures that any continued fraction completely specifies a single rational tangle. The tangle shown in Figure~\ref{fig:RatTangle1}(a) has continued fraction expansion $\infty = [0,0] = 0+\frac{1}{0}$. 

\begin{prop}[Conway \cite{conway1970enumeration}]
  \label{prop:Conway}
Equivalence classes of rational tangles are in one--to--one correspondence with the set $\QQ \cup \infty$. In particular, tangles $T(a_n, \dots, a_1)$ and $T(b_m,\dots, b_1)$ are equivalent if and only if the continued fractions $[a_n, \dots, a_1]$ and $[b_m, \dots, b_1]$ are equal.
\end{prop}

Using Proposition~\ref{prop:Conway}, we can put all our tangles into nice form. In particular, if a rational tangle corresponds to a positive rational number, then we can ensure its continued fraction expansion consists only of nonnegative integers. Similarly, if the tangle corresponds to a negative rational number, we can ensure the continued fraction expansion consists of nonpositive numbers. Thus in this paper, positive tangles have only positive crossings, and negative tangles have only negative crossings. 
This proves that we may divide all non-trivial rational tangles into two groups: \emph{positive tangles} and \emph{negative tangles}. In either case, the tangle has an alternating diagram.

In addition, we may require for a continued fraction with $n$ integers that $a_i \neq 0$ for all $i < n$. 

In the description of building rational tangles, we added vertical bands of crossings by rotating the points NW and NE, inserting the vertical band on the north of the tangle. Notice that we could have rotated SW and SE instead, adding a vertical band of crossings on the south of the tangle. These two methods are equivalent by a sequence of flypes. Likewise, we may add each horizontal band of crossings either on the west side of the tangle (by rotating the points NW and SW), or on the east side of the tangle (by rotating the points NE and SE). The following definition ensures a consistent choice. 

\begin{definition} 
\begin{enumerate}
\item[(a)] If $T$ is a positive tangle, then an alternating diagram for $T$ is \emph{admissible} if all the vertical bands of crossings were added by rotating the points NW and NE, and all the horizontal bands of crossings were added by rotating the points NE and SE. See Figure~\ref{fig:admissibleTangle}, left.
\item[(b)] If $T$ is a negative tangle, then an alternating diagram for $T$ is \emph{admissible} if all the vertical bands of crossings were added by rotating the points NW and NE, and all the horizontal bands of crossings were added by rotating the points NW and SW. See Figure~\ref{fig:admissibleTangle}, right.
\end{enumerate}
\label{def:admissible}
\end{definition}

\begin{figure}
  \centering
\begin{tabular}{ccc}
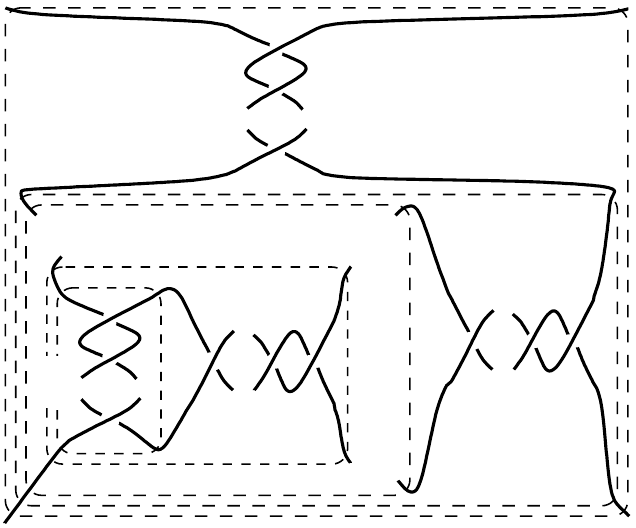 & \hspace{.1in} &
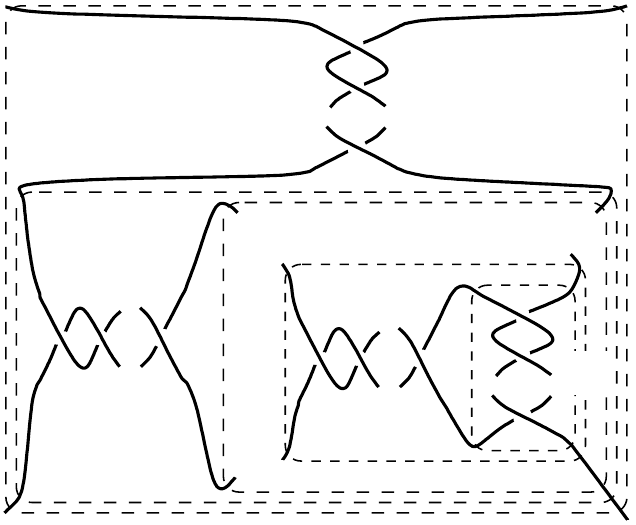
\end{tabular}
\caption{Shown are general forms of admissible tangles, positive on left, negative on right, for $n$ even, and $a_n=0$. (For $n$ odd, the band of $a_1$ crossings will be horizontal, $a_2$ vertical, etc. The band of $a_{n-1}$ crossings will be vertical in all cases.)}
\label{fig:admissibleTangle}
\end{figure}

By a sequence of flypes, any non-trivial tangle has an admissible diagram.

\subsection{Montesinos Links}
\label{subsec:montesinoslinks}

Recall that the \emph{numerator closure} $\num(T)$ of a tangle $T$ is formed by connecting NW to NE and SW to SE by simple arcs with no crossings. The \emph{denominator closure} $\denom(T)$ is formed by connecting NW to SW and NE to SE by simple arcs with no crossings.

Given two rational tangles $T_1$ and $T_2$ with slopes $q_1$ and $q_2$, we form their \emph{sum} by connecting the NE and SE corners of $T_1$ to the NW and SW corners of $T_2$, respectively, with two disjoint arcs. If $q_1$ or $q_2$ is an integer, then the sum  $T_1 + T_2$ is also a rational tangle; this is called a \emph{trivial} sum. 

The \emph{cyclic sum} of $T_1, \hdots, T_r$ is the numerator closure of the sum $T_1 + \hdots + T_r$.

\begin{definition}\label{def:Montesinos}
A \emph{Montesinos link} is the cyclic sum of a finite ordered list of rational tangles $T_1, \hdots, T_r$. See Figure~\ref{fig:MonLink}.
\end{definition}

\begin{figure}
\centering
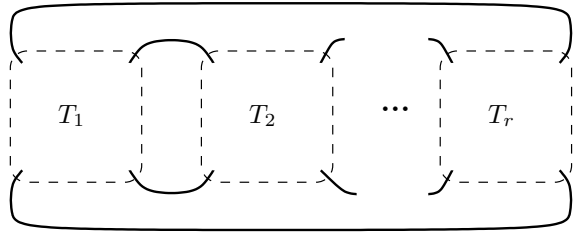
\caption{A Montesinos link with $r$ tangles}
\label{fig:MonLink}
\end{figure}

A Montesinos link is determined by the integer $r$ and an $r$-tuple of slopes $q_1$, $\hdots$, $q_r$, with $q_i \in \Q \cup \infty$. Note that if $q_i = \infty$ for some $i$, the link is disconnected. If $q_i$ is an integer, then it consists of a single band of horizontal crossings, which can be subsumed into an adjacent rational tangle in a sum. 
Thus we assume that $q_i \notin \Z \cup \infty$, to avoid trivial or disconnected sums.

\begin{thm}[Theorem~12.8 of \cite{bonahon2010}]\label{thm:bonahon}
Let $K$ be a Montesinos link obtained as the cyclic sum of $r \geq 3$ rational tangles whose slopes are $q_1, \hdots ,q_r \in \Q \setminus \Z$. Then $K$ is determined up to isomorphism by the rational number $\sum_{i=1}^{r}q_i$ and the vector $((q_1 \mod1), \hdots,(q_r \mod 1))$, up to dihedral permutation.
\end{thm} 
 
Note that this theorem gives isomorphism up to dihedral permuation; however we will only use isomorphism up to cyclic permutation. By Theorem~\ref{thm:bonahon}, given $K$ as the cyclic sum of $T_1, \hdots, T_r$, we can ``combine'' the integer parts of $q_1, \hdots,q_r$. The following definition makes use of this fact. 
 
\begin{definition}\label{def:reduced}
A diagram $D(K)$ is called a \emph{reduced Montesinos diagram} if it is the cyclic sum of the diagrams $T_i$, and for each $i$, the diagram of $T_i$ has either all positive or all negative crossings, and either
\begin{enumerate}
\item[(1)] all the slopes $q_i$ of tangles $T_i$ have the same sign, or
\item[(2)] $0<|q_i|< 1$ for all $i$.
\end{enumerate}
\end{definition}
 
It is not hard to see that every Montesinos link with $r \geq 3$ has a reduced diagram. For example, if $q_i < 0$ while $q_j > 1$, one may add $1$ to $q_i$ and subtract $1$ from $q_j$. By Theorem~\ref{thm:bonahon}, this does not change the link type. One may continue in this manner until condition (1) of Definition \ref{def:reduced} is satisfied.

We make one more definition.

\begin{definition}\label{def:admissibleMont}
A diagram $D(K)$ of the cyclic sum of $T_1, \hdots, T_r$ is an \emph{admissible Montesinos diagram} if the diagram of $T_i$ is an admissible tangle diagram for each $i$.
\end{definition}

Since every tangle has an admissible diagram, every Montesinos link with $r\geq 3$ has a reduced, admissible diagram.

We will need to know that an admissible diagram of a Montesinos link is prime. Recall that a diagram is \emph{prime} if, for any simple closed curve $\gamma$ meeting the diagram graph transversely in exactly two edges, the curve $\gamma$ bounds a region of the projection plane with no crossings. 

\begin{prop}\label{prop:montesinosPrime}
A reduced admissible diagram of a Montesinos link with at least two (non-trivial) tangles is prime. 
\end{prop}

We will prove Proposition~\ref{prop:montesinosPrime} using two lemmas.

\begin{lemma}\label{lemma:tangleprime}
If $T$ is a reduced admissible diagram of a rational tangle with at least two crossings, then either $T$ is a single vertical band of crossings and $\denom(T)$ is prime, or $\num(T)$ is prime.
\end{lemma}

\begin{proof}
If $T$ is a rational tangle with only a single vertical band of crossings, then the denominator closure of $T$ is a $(2,q)$--torus link, with $q>1$ by assumption on the number of crossings. Otherwise, the numerator closure of $T$ will be an alternating diagram of a 2--bridge link. By work of Menasco, in either case the diagram will be prime \cite{menasco:incompress}.
\end{proof}

\begin{lemma}\label{lemma:tangleSumPrime}
Suppose $T_1$ is a diagram of a connected non-trivial tangle such that either $\num(T_1)$ or $\denom(T_1)$ is prime, and suppose that $T_2$ is a reduced admissible diagram of a rational tangle with at least one crossing. Then $\num(T_1+T_2)$ is prime.
\end{lemma}

\begin{proof}
Let $D$ be the diagram of $\num(T_1+T_2)$ and suppose that $\gamma$ is a simple closed curve meeting $D$ exactly twice.

\textbf{Case 1.} The curve $\gamma$ meets $D$ outside of both tangles. Then since the diagram has no crossings outside the two tangles, either $\gamma$ bounds a portion of the diagram with no crossings on one side, or $\gamma$ encloses $T_1$ on one side, $T_2$ on the other. But in the latter case, $\gamma$ would have to meet $D$ four times, contradicting the fact that it meets $D$ exactly twice.

\textbf{Case 2.} The curve $\gamma$ meets $D$ twice in the tangle $T_1$. Then we may isotope $\gamma$ to be contained entirely in $T_1$, that is, within the Conway sphere enclosing $T_1$. Then $\gamma$ can be drawn into the numerator and denominator closures of $T_1$. One of these is prime; without loss of generality say $\num(T_1)$ is prime (otherwise rotate the diagram for the following argument). The curve $\gamma$ must contain no crossings on one side. If it contains no crossings in its interior, then there are no crossings in the interior of $\gamma$ in $D$, and we are done. So suppose $\gamma$ contains no crossings on its exterior in $\num(T_1)$. Then the tangle must contain all its crossings on the interior of $\gamma$. Moreover, exactly two strands of the tangle run to the exterior of $\gamma$, and four strands must connect from the NW, NE, SE, and SW corners. See Figure~\ref{fig:primeTangle}. This is impossible for a connected tangle.

\begin{figure}
\includegraphics{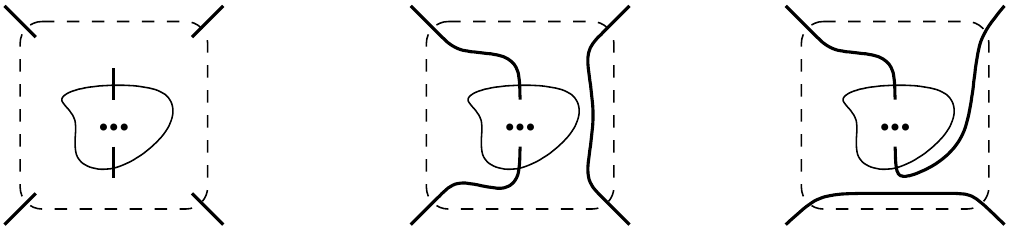}
  \caption{If $\gamma$ has no crossings to the exterior of $\num(T_1)$, then $T_1$ cannot have a connected diagram}
  \label{fig:primeTangle}
\end{figure}

\textbf{Case 3.} The curve $\gamma$ meets $D$ twice in the tangle $T_2$. If $T_2$ has prime numerator or denominator closure, then the same argument as in Case~2 applies, to guarantee that $\gamma$ bounds no crossings on one side. By Lemma~\ref{lemma:tangleprime}, the only remaining case is that $T_2$ consists of a single crossing. But then if $\gamma$ meets $D$ exactly twice in a tangle consisting of a single crossing, it cannot encircle that crossing, but must bound a region of the diagram with no crossings to the interior.

\textbf{Case 4.} The curve $\gamma$ intersects $T_1$ exactly once. The tangle $T_1$ is bounded by a square meeting the diagram in four points, NW, NE, SE, and SW.

If $\gamma$ exits $T_1$ by running through adjacent sides of the square, then we may form a new simple closed curve $\gamma'$ meeting the diagram exactly twice by taking $\gamma$ inside the square, and taking portions of the two sides of the square meeting in one of the corners (NW, NE, SE, or SW). This new curve $\gamma'$ can be drawn into the diagrams of $\num(T_1)$ and $\denom(T_1)$. Since once of these is prime, without loss of generality $\num(T_1)$, $\gamma'$ bounds no crossings on one side in that diagram. If the interior, then to the interior $\gamma'$ bounds a single strand of the diagram, and we may slide $\gamma'$ and $\gamma$ along this strand to remove the intersection of $\gamma$ with $T_1$. If $\gamma'$ bounds no crossings to the exterior, then since there are three knot strands to the exterior, emanating from three of NW, NE, SE, SW, the tangle diagram is not connected. This is a contradiction.

If $\gamma$ exits $T_1$ by running through north and south sides of the square, then consider the portion of $\gamma$ inside $T_1$, and form $\denom(T_1)$. We may connect north to south in $\denom(T_1)$ by an arc that does not meet the diagram of $\denom(T_1)$. Connecting this to $\gamma$, we obtain a closed curve meeting the diagram of $\denom(T_1)$ exactly once. This is impossible. Note this argument did not need $\denom(T_1)$ to be prime. Symmetrically, if $\gamma$ exits $T_1$ by running through the east and west, then we may form a closed curve meeting $\num(T_1)$ exactly once, which is impossible. So we may assume Case 4 does not happen. 

\textbf{Case 5.} The curve $\gamma$ intersects $T_2$ exactly once. Again if $T_2$ contains more than one crossing, Lemma~\ref{lemma:tangleprime} and the argument of Case 4 will imply we can isotope $\gamma$ outside of $T_2$. If $T_2$ contains exactly one crossing, and $\gamma$ meets $T_2$ exactly once, then it must meet $T_2$ in one of the strands running from NE, NW, SW, SE to the center, and we may isotope it from that point of intersection to the corner without meeting any crossings. Thus we may assume, after isotopy, that Case 5 does not happen.

Thus in all cases, $\gamma$ bounds no crossings on one side. 
\end{proof}

\begin{proof}[Proof of Proposition~\ref{prop:montesinosPrime}]
The proof is by induction on the number of tangles in the Montesinos link. If there are two tangles, then either the result holds by Lemma~\ref{lemma:tangleSumPrime}, or both tangles consist of a single crossing. In that case, their sum is a horizontal band of two crossings, hence the Montesinos link is a standard diagram of a $(2,2)$--torus link, which is prime.

Now suppose that any reduced admissible diagram of a Montesinos link with $k$ tangles is prime, and consider a Montesinos link with $k+1$ tangles. The first $k$ tangles to the right have a sum satisfying the hypotheses on $T_1$ in Lemma~\ref{lemma:tangleSumPrime}, and the $(k+1)$-st tangle satisfies the hypothesis on $T_2$. So by that lemma, the diagram of the Montesinos link is prime.
\end{proof}

\subsubsection{Montesinos Links of Interest}
\label{subsubsec:MontesinosLinksofInterest}

Futer, Kalfagianni, and Purcell found a volume estimate for Montesinos links with at least three positive or three negative tangles~\cite{FKP}. A Montesinos link with only one or two tangles has an alternating diagram; its volume is bounded by Lackenby~\cite{lackenby}. Thus the following are the only types of Montesinos links whose volumes cannot be estimated by previous results:
\begin{enumerate}
\item[(a)] Montesinos links with two positive and one negative tangles,
\item[(b)] Montesinos links with one positive and two negative tangles,
\item[(c)] Montesinos links with two positive and two negative tangles.
\end{enumerate}

Notice that the mirror image of a type (b) link is a type (a) link; and taking the mirror will not change the volume of the link complement. Thus we will ignore type (b) links in favor of type (a) links in our analysis. Notice also that there is only one ``arrangement'' of a type (a) link, up to cyclic permutation. However there are two arrangements of type (c) links.
 
\begin{definition} 
A \emph{$++-$ link} is Montesinos link which is the cyclic sum of  $T_a, T_b, T_c$, where $T_a$ and $T_b$ are positive tangles and $T_c$ is a negative tangle. A \emph{$+-+-$ link} is a Montesinos link which is the numerator closure of the sum $T_a + T_b + T_c + T_d$, where $T_a$ and $T_c$ are positive tangles and $T_b$ and $T_d$ are negative tangles. A \emph{$++--$ link} is a Montesinos link which is the numerator closure of the sum $T_a + T_b + T_c + T_d$ where $T_a$ and $T_b$ are positive tangles and $T_c$ and $T_d$ are negative tangles.
\label{def:++-links}
\end{definition}

Our goal is to find volume bounds for these types of Montesinos links. We take Definition \ref{def:++-links} as the definition not only of $++-$, and $+-+-$ links, but also of the tangles $T_a$, $T_b$, $T_c$ and $T_d$. Notice that the definitions of $T_a$, $T_b$, and $T_c$ change depending on whether we are talking about $++-$, $+-+-$, or $++--$ links.

\begin{remark}\label{remark:++--}
In fact, we can actually consider only one of $++--$ and $+-+-$ links, as follows. Since we are assuming the link is reduced, each of our tangles has slope with absolute value at most $1$, as in Definition~\ref{def:reduced}(2). Thus in a $++--$ link, we may subtract $1$ from the second slope and add $1$ to the third. By Theorem~\ref{thm:bonahon}, the result will be equivalent to a $+-+-$ link. For this reason, we only consider $+-+-$ links.
\end{remark}

Notice that reduced diagrams for  $++-$ and $+-+-$ links do not satisfy part (1) of Definition~\ref{def:reduced}; therefore they must satisfy (2). This means that there are no integer parts of slopes of each tangle. In other words, if $T_a$ has slope $[a_n, \dots, a_1]$, then we may assume that $a_n = 0$. Recall from subsection~\ref{sec:rationaltangles} that $a_n$ corresponds to a horizontal band of crossings; so $a_{n-1}$ corresponds to a vertical band of crossings, as in Figure~\ref{fig:admissibleTangle}. Recall also that $a_i \neq 0$ for $i \neq n$. Thus the tangles $T_a$, $T_b$, $T_c$ and $T_d$ have the form of Figure~\ref{fig:admissibleTangle}, except possibly $n$ even replaced with $n$ odd, meaning the vertical band of $a_1$ crossings will be horizontal.

\section{Estimating the guts}\label{sec:estimatingtheguts}

Our volume bounds use estimates developed by Futer, Kalfagianni, and Purcell to bound volumes of semi-adequate links \cite{FKP}. We will see that the $++-$ and $+-+-$ Montesinos links of interest are semi-adequate, and so they fit into this machinery. In this section, we recall the definition of semi-adequate links, and review the relevant features of \cite{FKP}. All the necessary details that we need from \cite{FKP} are contained in this paper. However, one may consult \cite{FKP} or the survey article \cite{fkp:survey} for additional information and for the proofs of the results that we cite. 
 
\subsection{Semi-adequate links}

Given a link diagram $D$ and a crossing $x$ of $D$, we define a new diagram by replacing the crossing $x$ with a crossing--free \emph{resolution}. There are two ways to resolve a crossing, shown in Figure~\ref{fig:resolutions}: the $A$--resolution or the $B$--resolution. 

\begin{figure}
\centering
\includegraphics{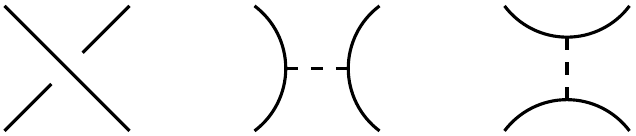}
\caption{Left to right: A crossing, its $A$--resolution, its $B$--resolution }
\label{fig:resolutions}
\end{figure}

\begin{definition} 
A \emph{state} $\sigma$ is a choice of $A$-- or $B$--resolution at each crossing of a diagram $D$. Applying a state $\sigma$ to a diagram $D$ yields a collection of crossing--free simple closed curves called \emph{state circles}. If we attach an edge to the state circles at each removed crossing, i.e.\ attach the dashed edge shown in Figure~\ref{fig:resolutions}, we obtain a tri-valent graph $H_\sigma$. The edges coming from crossings, dashed in the figure, are called \emph{segments}.
\label{def:state} 
\end{definition}

We are concerned mainly with the \emph{all--$A$ state}, which chooses the $A$--resolution at each crossing. We will occasionally mention the all--$B$ resolution as well. The graph \emph{$H_A$} is obtained by applying the all--$A$ state to $D$ and including segments. For an example, see Figure~\ref{fig:diagrams}.
 
\begin{figure}
\centering
\begin{tabular}{cccc}
\includegraphics[height=1in]{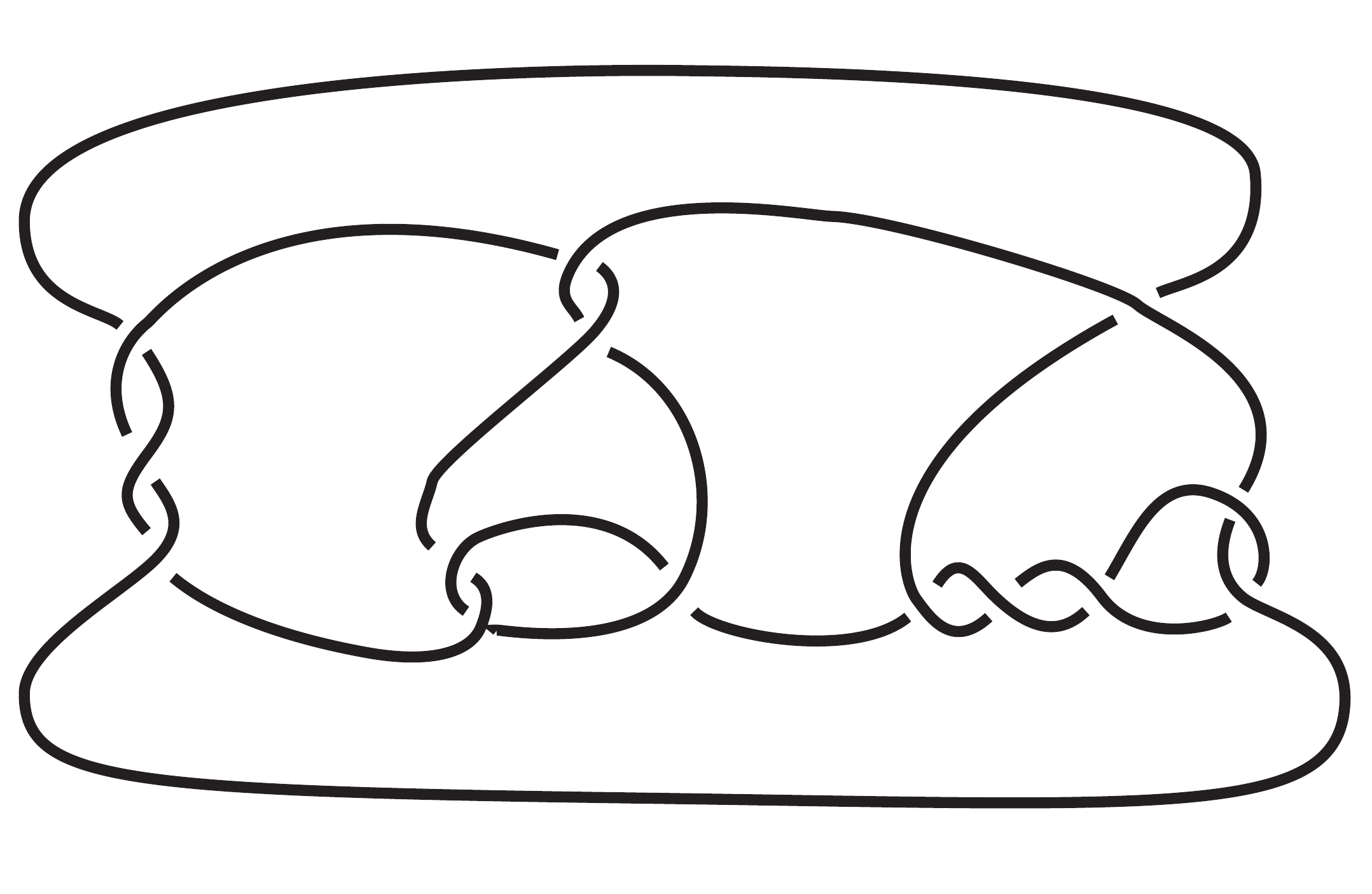} &
\includegraphics[height=1in]{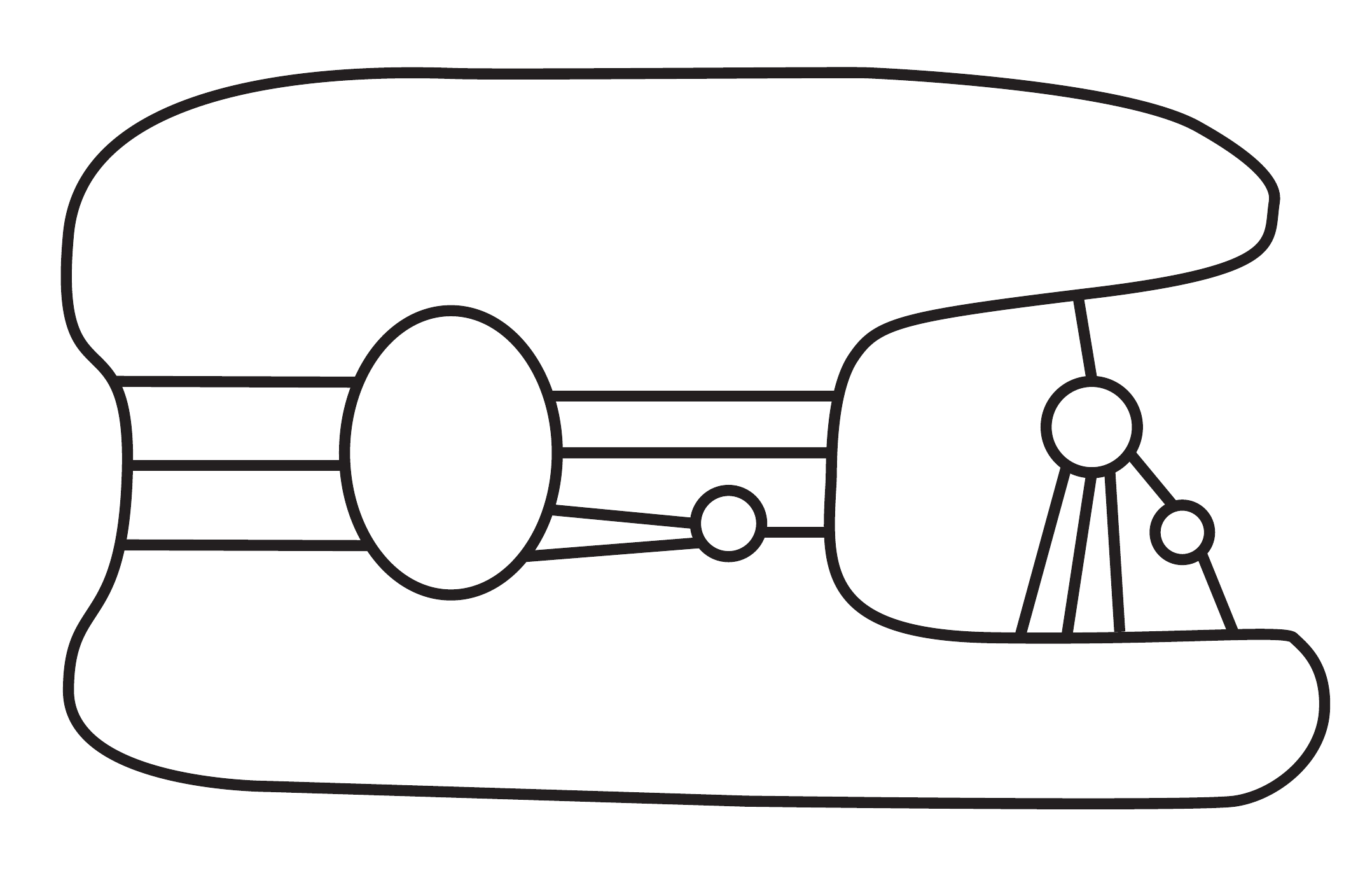} &
\includegraphics[height=1in]{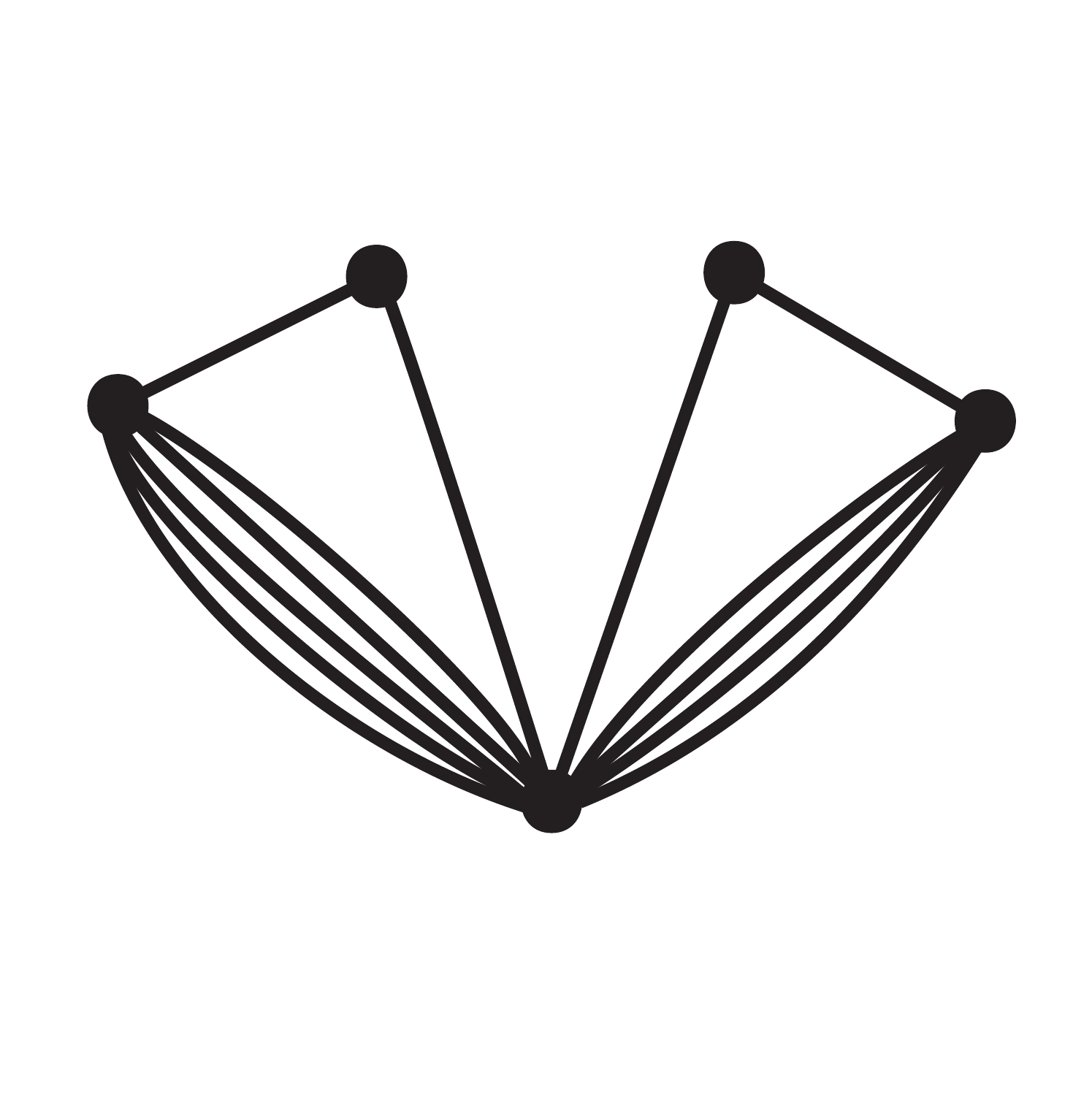} &
\includegraphics[height=1in]{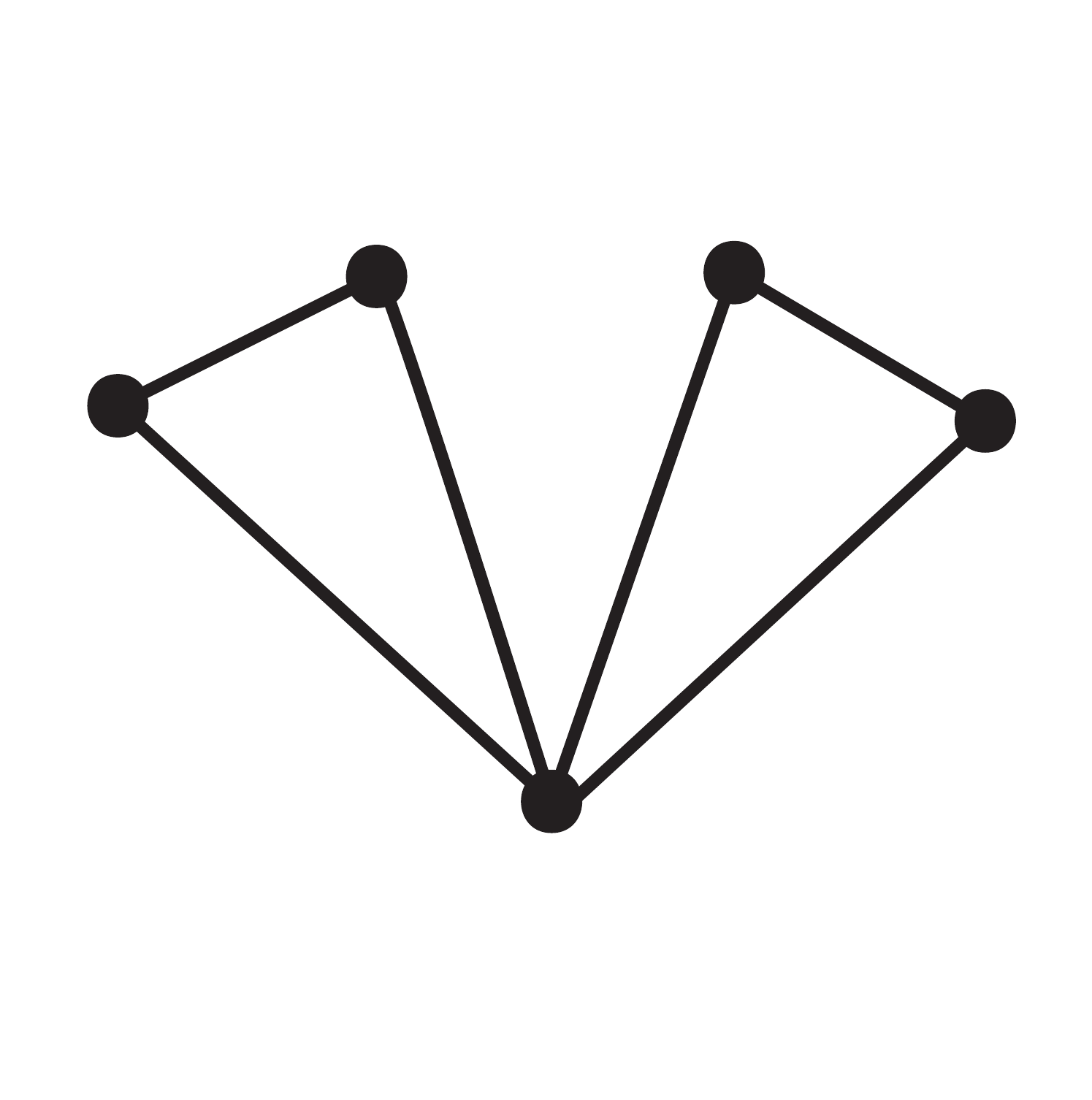} \\
(a) & (b) &
(c) & (d) 
\end{tabular}
\caption{(a) A $++-$ Montesinos link, (b) its graph $H_A$, (c) the state graph $\G_A$, (d) the reduced state graph $\G_A'$}
\label{fig:diagrams}
\end{figure}

\begin{definition}
From $H_A$ we create the \emph{$A$--state graph} $\G_A$ by shrinking each state circle to a single vertex. We obtain the \emph{reduced $A$--state graph} $\G'_A$ by removing multiple edges between pairs of vertices in $\G_A$. An example is shown Figure~\ref{fig:diagrams}.
\label{def:G_A}
\end{definition}

The following two lemmas concern $H_A$ for admissible diagrams of Montesinos links, and they follow immediately from the structure of admissible tangles, as in Figure~\ref{fig:admissibleTangle}. Both lemmas are illustrated in Figure~\ref{fig:HAposnegtangles}.

\begin{lemma}\label{lemma:HAPosAdmissibleTangle}
Let $T$ be a positive admissible tangle with corresponding continued fraction $[0, a_{n-1}, \dots, a_2, a_1]$. Then for any Montesinos knot containing $T$, the graph $H_A$ will have the following properties in a neighborhood of $T$.
\begin{enumerate}
\item A portion of a state circle, call it $S_0$, runs from NW to SW, and a portion of another, call it $S_{n-1}$, runs from NE to SE. 
\item There are $a_{n-1}>0$ horizontal segments running from $S_0$ to $S_{n-1}$ at the north of the graph.
\item For each $a_i$ with $i \equiv n (\mbox{mod } 2)$, there exists a horizontal string of $a_i$ state circles alternating with $a_i$ segments, with the segment on the far east having one endpoint on $S_{i+1}$, south of any other segments, and with the final state circle on the far west denoted by $S_{i-1}$.
\item For each $a_i$ with $i \equiv (n-1) (\mbox{mod } 2)$, there are $a_i$ horizontal segments connecting $S_0$ and~$S_i$.\qed
\end{enumerate}
\end{lemma}

\begin{lemma}\label{lemma:HANegAdmissibleTangle}
Let $T$ be a negative admissible tangle with corresponding continued fraction $[0, a_{n-1}, \dots, a_2, a_1]$. Then for any Montesinos knot containing $T$, the graph $H_A$ will have the following properties in a neighborhood of $T$.
\begin{enumerate}
\item A portion of a state circle, call it $S_n$, runs from NW to NE, and a portion of another, call it $S_0$, runs from SW to SE.
\item For each $a_i$ with $i\equiv (n-1) (\mbox{mod } 2)$, there exists a vertical string of $a_i$ state circles alternating with $a_i$ segments, with the segment at the far north having an endpoint on $S_{i+1}$, to the east of all other segments on $S_{i+1}$. The final state circle at the far south is denoted by $S_{i-1}$.
\item For each $a_i$ with $i \equiv n (\mbox{mod } 2)$, there exist $a_i$ vertical segments connecting $S_0$ and $S_i$.\qed
\end{enumerate}
\end{lemma}

\begin{figure}
\begin{center}
\begin{tabular}{ccc}
  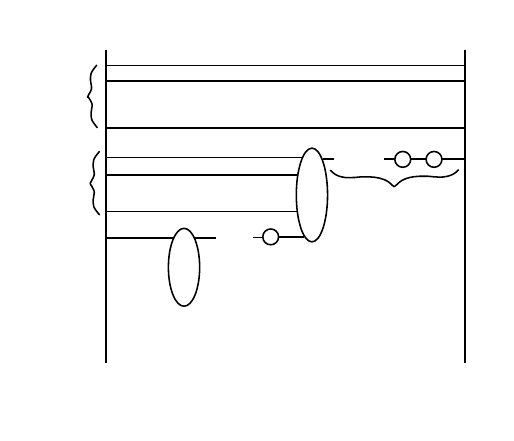 & \hspace{.5in} &
  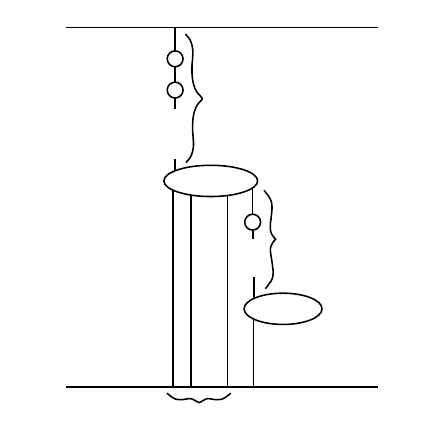 \\
  (a) Positive & & (b) Negative \\
\end{tabular}
\end{center}
  \caption{General form of $H_A$ in neighborhood of an admissible positive tangle, left, and an admissible negative tangle, right.}
  \label{fig:HAposnegtangles}
\end{figure}

The following definition is due to Lickorish and Thistlethwaite.

\begin{definition}
A link diagram $D(K)$ is called \emph{$A$--adequate} if $\G_A$ has no 1--edge loops, and \emph{$B$--adequate} if $\G_B$ has no 1--edge loops.
\end{definition}
 
\begin{thm}[Lickorish and Thistlethwaite, \cite{lickorish1988}]\label{thm:lickorish}
Let $D(K)$ be a reduced Montesinos diagram with $r > 0 $ positive tangles and $s > 0$ negative tangles. Then $D(K)$ is $A$--adequate if and only if $r \geq 2$ and $B$--adequate if and only if $s \geq 2$. Since $r+s \geq 3$ in a reduced diagram, $D$ must be either $A$-- or $B$--adequate.
\end{thm}
 
Also note that if $r=0$ or $s=0$ then $D(K)$ is alternating, in which case it is both $A$-- and $B$--adequate. We will see shortly why adequacy is a desirable property.

From $H_A$ we may obtain a surface as follows. The state circles of $H_A$ bound disjoint disks in the $3$--ball below the projection plane. To these disks, attach a half--twisted band corresponding to each crossing in the original diagram. 
This forms a connected surface called the \emph{$A$--state surface}, or simply $S_A$.

\begin{thm}[Ozawa, \cite{ozawa2011}]
Let $D$ be a (connected) diagram of a link $K$. Then the surface $S_A$ is essential in $S^3 \setminus K$ if and only if $D$ is $A$--adequate.
\label{thm:essential}
\end{thm}

Define the manifold with boundary $M_A = S^3 \cutalong S_A$, where $S^3 \cutalong S_A$ is defined to be $S^3$ with a regular neighborhood of $S_A$ removed. 

\begin{definition}
Let $M = S^3 \setminus N(K)$, and define the \emph{parabolic locus} of $M_A$ to be $P = \partial M_A \cap \partial M$. The parabolic locus consists of annuli. These annuli consist of the remnants of the knot diagram.
\label{def:paraboliclocus}
\end{definition}

In \cite{FKP}, it was shown that $M_A$ can be cut into ideal polyhedra. We will not describe the details of this cutting here, because we will not need those details; we will concern ourselves only with the results. The cutting produces finitely many polyhedra that lie below the projection plane, and a single polyhedron above, which we call the \emph{upper polyhedron}. In this paper, we only need to study the upper polyhedron. It has a nice combinatorial description coming from the graph $H_A$, which we now recall. 

To visualize the upper polyhedron, start with the state graph $H_A$. Recall that $H_A$ lies in the projection plane, and is composed of state circles and segments. We call a given state circle $S$ \emph{innermost} if $S$ bounds a region in the projection plane which does not contain any segments of $H_A$. We shade each innermost disk, giving each a different color. These will correspond to the distinct \emph{shaded faces} of the upper polyhedron. See Figure~\ref{fig:ppm-Shaded}(a) for an example.

\begin{figure} 
\centering
\begin{tabular}{ccc}
  \includegraphics[height=1in]{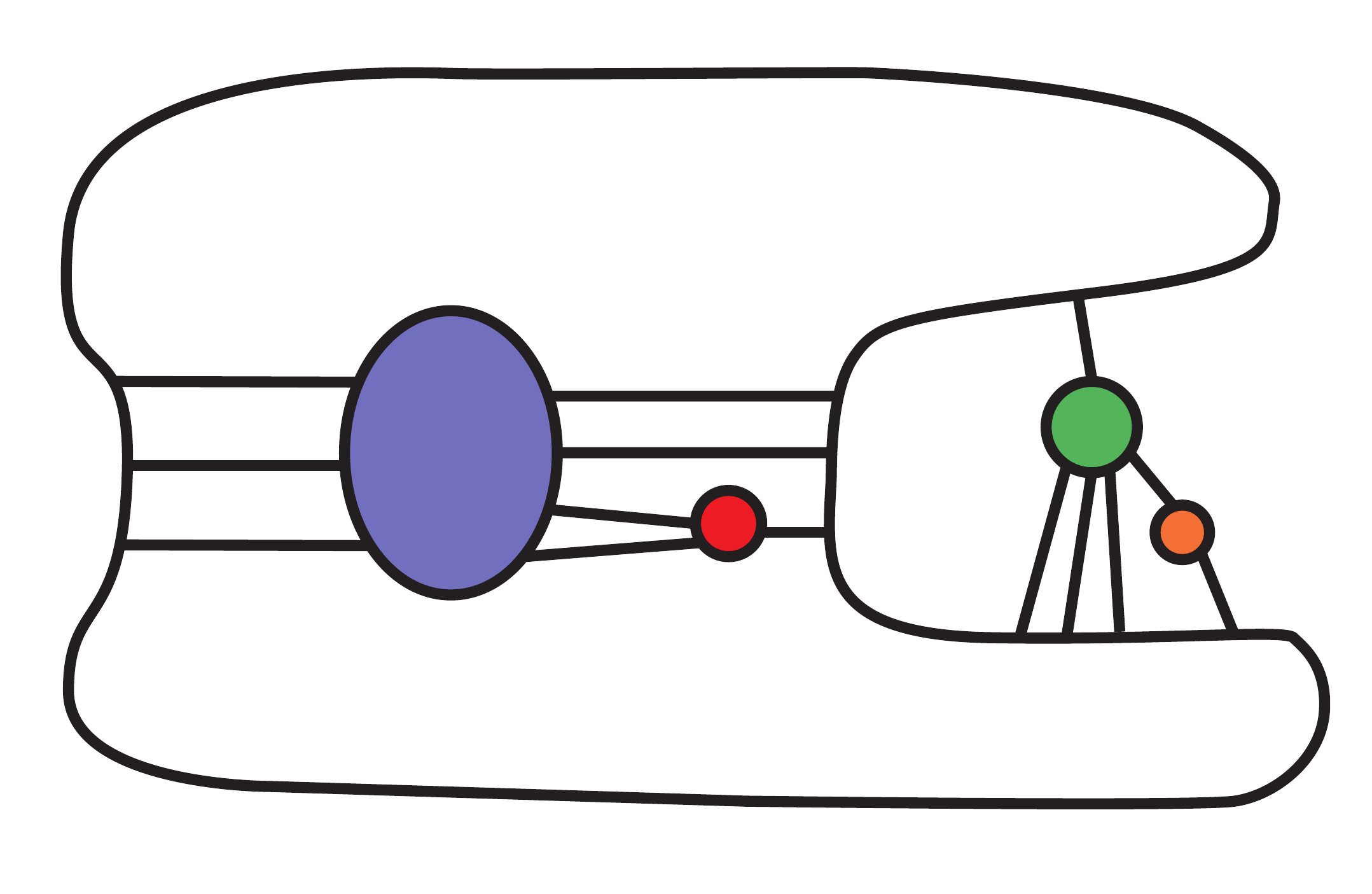} &
  \includegraphics[height=1in]{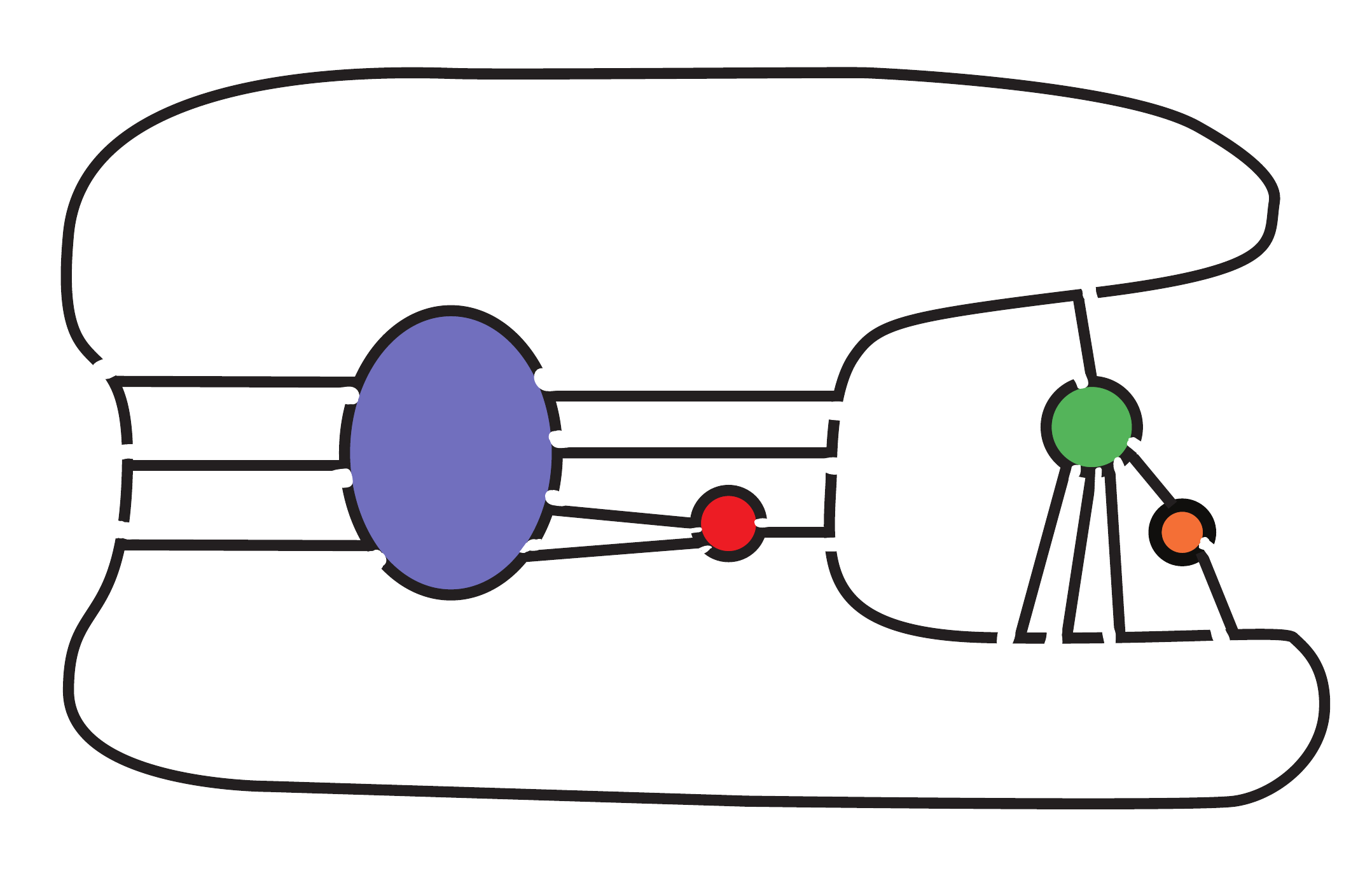} &
  \includegraphics[height=1in]{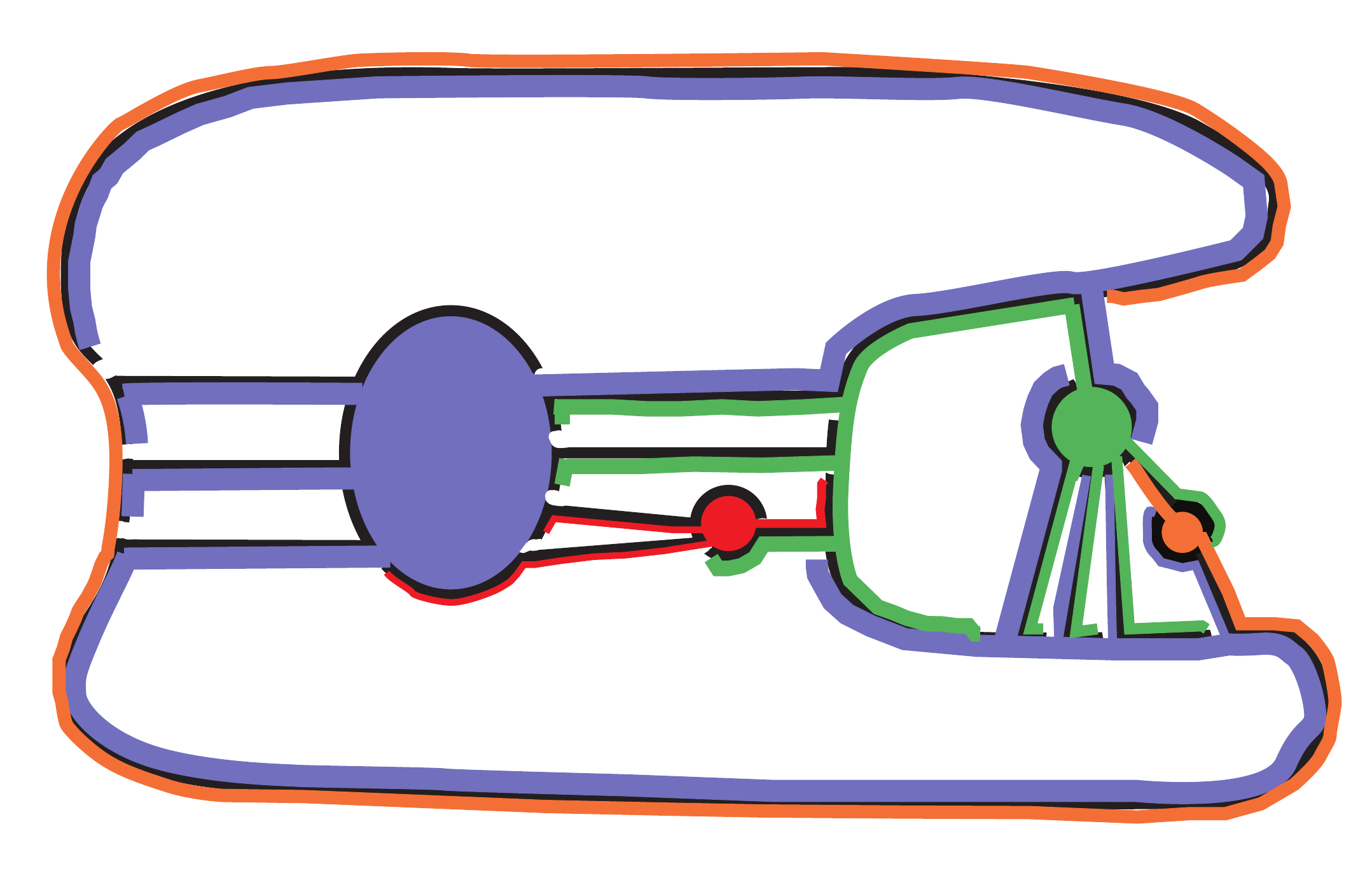} \\
  (a) & (b) & (c) \\
\end{tabular}
\caption{Building the upper polyhedron. (a) Shade innermost disks. (b) Remove portions of state circle. (c) Construct tentacles.}
\label{fig:ppm-Shaded}
\end{figure}

The faces extend from the innermost state circles as follows. Given a segment $s$ of $H_A$, rotate $H_A$ so that $s$ is vertical. There are two distinct ways to perform this rotation; the procedure that follows is independent of that choice. Once $s$ is vertical, erase a small part of the graph immediately northeast of $s$ and a small part immediately southwest of $s$. Repeat this rotation and erasing for each segment in the graph. See, for example, Figure~\ref{fig:ppm-Shaded}(b).

Finally, draw the ``tentacles:'' Choose a segment $s$ that meets one of the innermost state circles. The innermost state circle bounds a shaded face. Rotate $H_A$ so that $s$ is vertical with the shaded face on the top. The small hole to the northeast of $s$ acts as a ``gate'', allowing the shaded face to run through the hole, forming a \emph{tentacle}. The tentacle runs in a thin band along $H_A$, adjacent to a segment and a state circle, running south and then east. It terminates when it runs into a segment. However, the tentacle may run \emph{past} other segments on the opposite side of the state circle without terminating. When this occurs, the tentacle spawns a new tentacle, running through the hole in $H_A$ adjacent to that segment. Each new tentacle also terminates when it hits a segment, and also spawns other tentacles when it runs past a segment. Continue until each tentacle has terminated. Now one shaded face is complete. Repeat this process for each innermost disk. Figure~\ref{fig:ppm-Shaded}(c) shows a completed example of the upper polyhedron.

\begin{definition}
For a given shaded face, or a given tentacles of a shaded face, we will say the face or tentacle \emph{originated} in the innermost state circle of the same color. The place where a tentacle terminates by running into a segment is called the \emph{tail} of the tentacle. The place where a tentacle runs adjacent to a segment is called the \emph{head} of the tentacle.
\end{definition}

The upper polyhedron has faces that include the shaded faces, as well as white faces corresponding to unshaded regions of the diagram. Edges run from head to tail of tentacles, and separate white and shaded faces. Vertices are the remnants of $H_A$. In \cite{FKP}, it was proved that this process produces an ideal polyhedron, with ideal vertices on the parabolic locus.

\begin{lem}[Futer, Kalfagianni, and Purcell, Theorem~3.13 of \cite{FKP}]
\label{lem:checkerboardidealpolyhedron}
Let $D(K)$ be an $A$--adequate link diagram. Then the polyhedron $P_A$ as described above is a checkerboard colored ideal polyhedron with 4--valent vertices.
\end{lem}

Here ``checkerboard'' means that white faces never share an edge, and neither do colored faces. After erasing the small holes as in Figure~\ref{fig:ppm-Shaded}(b), each connected component of the remnants of $H_A$ is one ideal vertex of $P_A$. Each such component consists of segments and portions of state circle, with segments meeting two distinct shaded faces on opposite sides, and portions of state circle meeting white faces at the endpoints of the ideal vertices. This is why the vertices are 4--valent.

The careful reader may notice we have not discussed the \emph{non-prime arcs} in the polyhedral decomposition in \cite{FKP}. This is because in Montesinos links, non-prime arcs only occur between adjacent negative tangles. (See~\cite[Lemma~8.7]{FKP}). We focus on $++-$ and $+-+-$ links, which have no adjacent negative tangles. Thus these diagrams have no non-prime arcs.

We are now ready to discuss the upper polyhedron of the polyhedral decomposition of $++-$ and $+-+-$ Montesinos links.

\begin{lemma}\label{lemma:Pdecomp-ppm}
Let $K$ be a $++-$ Montesinos link. Then $K$ has a reduced, admissible diagram $D$ with upper polyhedron $P_A$ having the following properties.
\begin{enumerate}
\item There is an innermost disk, denoted $G$, between the two positive tangles $T_a$ and $T_b$.
\item There are $a_{n-1}>0$ segments running across $T_a$ from west to east, with east endpoints on $G$. Similarly, there are $b_{n-1}>0$ segments running across $T_b$ from west to east, with their west endpoints on $G$.
\item If $T_a$ contains no state circles in the interior, then all white faces in $T_a$ are bigons. Otherwise, there are $a_{n-1}$ bigons at the north of $T_a$, and a non-bigon just below. The same holds for $T_b$.
\item There is exactly one segment at the north of $T_c$. If $T_c$ has only one state circle, then all white faces below that state circle in $T_c$ are bigons.
\end{enumerate}
\end{lemma}

\begin{proof}
Immediate from Lemmas~\ref{lemma:HAPosAdmissibleTangle} and~\ref{lemma:HANegAdmissibleTangle}, and the definition of $P_A$.
\end{proof}
 
\begin{lem}\label{lemma:Pdecomp-pmpm}
Let $K$ be a $+-+-$ Montesinos link. Then $K$ has a reduced, admissible diagram $D$ with upper polyhedron $P_A$ having the following properties.
\begin{enumerate}
\item The tangle $T_b$ is contained in a state circle denoted $G$, lying between the two positive tangles.
\item There are $a_{n-1}>0$ segments running across $T_a$ from west to east, with east endpoints on $G$. Similarly, there are $c_{n-1}$ segments running across $T_c$ from west to east, with west endpoints on $G$.
\item If $T_a$ has no state circles in the interior, then all white faces are bigons. Otherwise, there are $a_{n-1}$ bigons at the north of $T_a$, and a non-bigon just below. Similarly for $T_c$.
\item There is exactly one segment at the north of $T_b$. If $T_b$ has only one state circle, then all white faces below that state circle in $T_b$ are bigons. Similarly for $T_d$.
\end{enumerate}
\end{lem}

\begin{proof}
Again immediate from Lemmas~\ref{lemma:HAPosAdmissibleTangle} and~\ref{lemma:HANegAdmissibleTangle}.
\end{proof}

\begin{remark}
In many of the figures below, for simplicity, we omit the step above of erasing a small portion of the graph near each segment, and just draw the tentacles without accurately portraying the ideal vertices of the diagram. 
\end{remark}

\subsection{Essential Product Disks}
By work of \cite{FKP}, the problem of bounding volumes of $A$--adequate links can be reduced to the problem of finding essential product disks in the upper polyhedron.

\begin{definition}
An \emph{essential product disk in the upper polyhedron} (EPD) is a properly embedded essential disk in $P_A$ whose boundary consists of two arcs in two shaded faces, and two points where the boundary meets the parabolic locus.
\end{definition}

We consider the boundary of an EPD, which we will draw into the diagram of $P_A$. Figure~\ref{fig:EPD}(a) depicts a portion of the diagram of $P_A$ with the boundary of an EPD. Given an EPD $E$ in $P_A$, we generally pull $\partial E$ into a \emph{normal square}. A normal square is a disk $D$ in normal form, with $\partial D$ intersecting $P_A$ in exactly four arcs. We will only need the special kinds of normal squares described in the next lemma.

\begin{figure}
\centering
\begin{tabular}{cc}
  \includegraphics{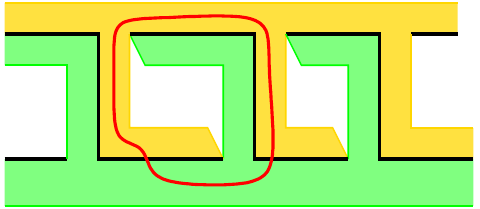} &
  \includegraphics{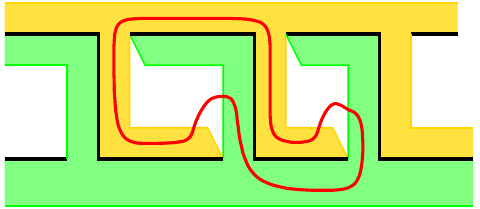} \\
  (a) An EPD & (b) Pulled into a normal square \\
\end{tabular}
\caption{A portion of the upper polyhedron containing an EPD}
\label{fig:EPD}
\end{figure}

\begin{lem}[Lemma 6.1 of \cite{FKP}]
Let $D$ be a diagram of a Montesinos link with an EPD, $E$, embedded in the upper polyhedron. Then $\partial E$ can be pulled off the parabolic locus to give a normal square in $P_A$ such that:
\begin{enumerate}
\item Two opposite sides of the square run through shaded faces, which we color green and gold, with endpoints of each side on distinct edges of $P_A$.
\item The other two edges of the square run through white faces, and cut off a single vertex of the white face. These are the \emph{white sides} of the normal square.
\item The single vertex of the white face that is cut off by a white side of $\partial E$ forms a triangle, so that when moving clockwise the edges of the triangle are colored gold-white-green.
\end{enumerate}
\label{lemma:pullnormalsquare}
\end{lem}

Given $E$ an EPD in $P_A$, two edges of the square run through shaded faces; we may choose how we color these faces green and gold. Then by  property~(3) of Lemma~\ref{lemma:pullnormalsquare}, we may force the triangles formed by the white faces to be oriented accordingly. Notice that the orientation given in property~(3) will force the white edges to occur at tails of gold tentacles and heads of green. Figure~\ref{fig:EPD}(b) depicts a normal square in $P_A$ oriented as in the above lemma.

\subsection{Complex essential product disks}
Recall that we are attempting to bound volumes of Montesinos links. A bound on volume is given by equation~\eqref{eqn:FKPVol}, originally from \cite[Theorem~9.3]{FKP}, which applies to reduced, admissible diagrams of hyperbolic Montesinos links. 
The term $||E_c||$ appears in this equation; we define it in this section. We also determine further information about EPDs in $P_A$ and their relationships to each other.

\begin{definition} Let $S$ be a surface in $P_A$. A \emph{parabolic compression disk} for $S$ is an embedded disk $E$ in $P_A$ such that:
\begin{enumerate}
\item[(i)] $E \cap S$ is a single arc in $\partial E$;
\item[(ii)] The rest of $\partial E$ is an arc in $\partial P_A$ that has endpoints disjoint from the parabolic locus $P$ and that intersects $P$ in exactly one transverse arc;
\item[(iii)] $E \cap S$ is not parallel in $S$ to an arc in $\partial S$ that contains at most one component of $S \cap P$.
\end{enumerate}
\label{def:pcdisk}
\end{definition}

Figure~\ref{fig:paraboliccompression}(a) shows a portion of a diagram of $P_A$. The red line represents the boundary of an EPD $D$. There is a parabolic compression disk $E$ for $D$. The dashed red line shows the arc as in part (ii)~of the above definition.

\begin{figure}
\centering
\begin{tabular}{cc}
\includegraphics{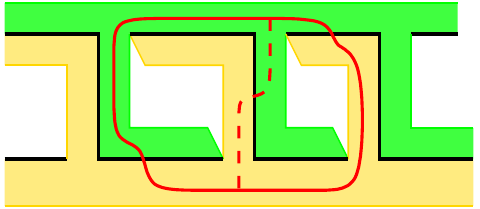} &
\includegraphics{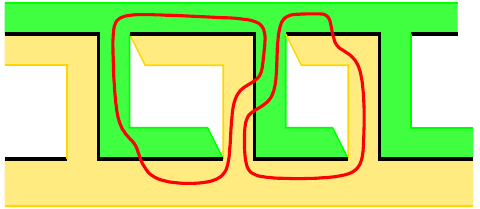} \\
(a) Arc of parabolic compression disk & (b) Two new EPDs
\end{tabular}
\caption{A parabolic compression}
\label{fig:paraboliccompression}
\end{figure}

\begin{definition}
If $D$ is an EPD in $P_A$ with a parabolic compression disk $E$ for $D$, then surgering $D$ along $E$ will produce a pair of new EPDs, $D'$ and $D''$. We say that $D$ and $D' \cup D''$ are \emph{equivalent under parabolic compression}.
\label{def:equivalentunderp}
\end{definition}

Figure~\ref{fig:paraboliccompression}(b) shows the two new EPDs $D'$ and $D''$ obtained from the disk $D$ in Figure~\ref{fig:paraboliccompression}(a). These new disks are equivalent to $D$ under parabolic compression.

\begin{definition}
An essential product disk $D$ in $P_A$ is called 
\begin{enumerate}
\item \emph{simple} if $D$ is the boundary of a regular neighborhood of a white bigon face of $P_A$,
\item \emph{semi-simple} if $D$ is equivalent under parabolic compression to a union of simple disks (but $D$ is not simple),
\item \emph{complex} if $D$ is neither simple nor semi-simple.
\end{enumerate}
\label{def:simplesemicomplex}
\end{definition}

We see that the EPDs in Figure~\ref{fig:paraboliccompression}(b) are simple, and the EPD in Figure~\ref{fig:paraboliccompression}(a) is semi-simple. The following lemma identifies semi-simple EPDs more generally, using a diagrammatic condition.

\begin{lem}\label{lem:semisimple}
Let $D$ be an essential product disk in $P_A$ that is not simple. Then $D$ is semi-simple if its boundary bounds a region in the projection plane that contains only bigon white faces of $P_A$.
\end{lem}

\begin{proof}
Suppose $D$ is an EPD in $P_A$ whose boundary contains only white bigon faces on one side, say the inside. The boundary of $D$ must run through two shaded faces with two colors, say green and gold, and it switches between colors by running over an ideal vertex $v$. Because there are only bigon white faces to the inside, the white face adjacent to $v$ must be a bigon. The bigon face must have edges also meeting green and orange shaded faces. Thus there is an arc with endpoints on $\partial D$ running through the ideal vertex of the bigon opposite $v$. See Figure~\ref{fig:semisimpleEPD}.

\begin{figure}
\centering 
\includegraphics{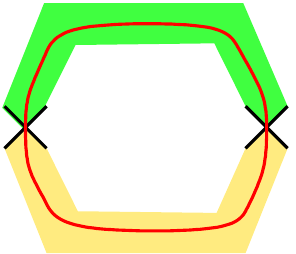}
\hspace{.2in}
\includegraphics{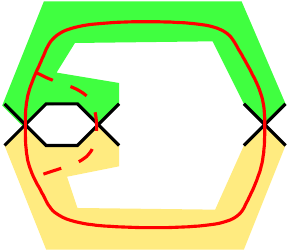}
\hspace{.2in}
\includegraphics{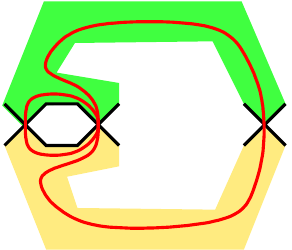}
\caption{An EPD bounding bigons to the interior, left, must have ideal vertex adjacent to a bigon, middle, which allows us to parabolically compress, right.}
\label{fig:semisimpleEPD}
\end{figure}
 
This arc defines a parabolic compression disk for $D$. Thus $D$ is equivalent under parabolic compression to two EPDs, one encircling the bigon, and one $D'$, whose boundary agrees with $D$ except excludes this bigon. Now $D'$ is an EPD bounding one fewer bigon in its interior. We may repeat the above process, compressing off another bigon. There are finitely many bigons in this region, so we continue until there is only one left. Then $D$ is equivalent under parabolic compression to the union of simple EPDs. 
\end{proof}

\begin{lem}[Lemma~5.8 of \cite{FKP}]
There exists a set $E_s \cup E_c$ of essential product disks in $P_A$ such that
\begin{enumerate}
\item $E_s$ is the set of all simple disks in $P_A$.
\item $E_c$ consists of complex disks. Further, $E_c$ is minimal in the sense that no disk in $E_c$ is equivalent under parabolic compression to a subcollection of $E_s \cup E_c$.
\item $E_c$ is also maximal in the sense that if any complex disks are added to $E_c$, then $E_c$ is no longer minimal.
\end{enumerate}
\label{lem:definitionE_c}
\end{lem}

We take Lemma~\ref{lem:definitionE_c} as a definition of the minimal set $E_c$. The number of EPDs in $E_c$ is denoted $||E_c||$. This is the term we wish to bound in Equation~\eqref{eqn:FKPVol}. It is not easy to tell directly from Definitions~\ref{def:pcdisk}, \ref{def:equivalentunderp}, and Lemma~\ref{lem:definitionE_c} whether a set of complex disks $E_c$ is minimal. The following lemma provides an easier way to characterize minimality.

\begin{lem}\label{lem:equivalentunderparaboliccompression}
Let $E_s$ be the set of all simple EPDs in $P_A$, and let $E_1$ and $E_2$ be EPDs in $P_A$. Then $E_2$ is equivalent under parabolic compression to a subcollection of $E_s \cup E_1$ if $\partial E_1$ and $\partial E_2$ differ only by white bigons, i.e.\ one of the regions $S$ bounded by $\partial E_1$ and one of the regions $T$ bounded by $\partial E_2$ are such that $(S\cup T) \setminus (S\cap T)$ contains no non-bigon white faces. 
\end{lem}

\begin{proof}
Consider first the case that $\partial E_1$ and $\partial E_2$ do not intersect.
If $S$ and $T$ are disjoint, then both $\partial E_1$ and $\partial E_2$ bound disjoint regions containing only white bigon faces, so both $E_1$ and $E_2$ are semi-simple. Thus $E_2$ is equivalent under parabolic compression to EPDs in $E_s$ alone, hence is equivalent to a subcollection of $E_s \cup E_1$.

If $\partial E_1$ and $\partial E_2$ do not intersect, but $S$ and $T$ are not disjoint, then $(S\cup T) \setminus (S\cap T)$ is an annular region of the projection plane between $\partial E_1$ and $\partial E_2$, and no non-bigon white faces lie in this region. This is shown in Figure~\ref{fig:disjoint1}(a).

\begin{figure}
\centering
\begin{tabular}{ccc}
\includegraphics{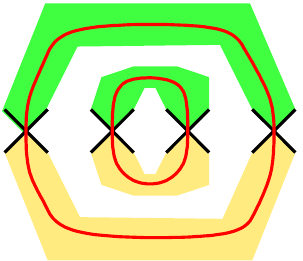} &
\includegraphics{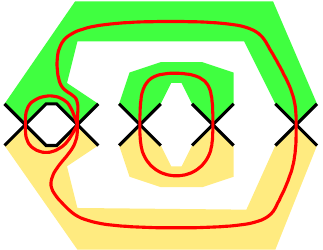} &
\includegraphics{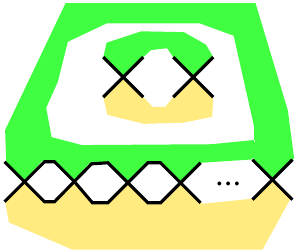} \\
(a) & (b) & (c)
\end{tabular}
\caption{Region between disjoint $\partial E_1$ and $\partial E_2$ conists of bigons.}
\label{fig:disjoint1}
\end{figure}

Focus on the left--most vertex, on $\partial E_1$. By Lemma~\ref{lem:checkerboardidealpolyhedron}, $P_A$ is checkerboard colored, hence the face to the right of this vertex is a white face. Therefore, either this face lies inside $\partial E_1$ and $\partial E_2$, and the two EPDs both meet this ideal vertex, or the face to the right must be a bigon. If the face is a bigon, we find a parabolic compression disk whose parabolic compression arc is the line running from $\partial E_1$ through the opposite vertex of the bigon. Therefore, $\partial E_1$ is parabolically compressible to a union of two EPDs, one of which is simple, as in Figure~\ref{fig:disjoint1}(b). Replace $E_1$ with the other EPD. Then we are in the same situation as before, only with one fewer bigon between boundaries of the EPDs. Repeat, compressing off bigons until there are no more. We obtain a finite chain of bigons, which must connect with one of the other three vertices in the figure. In Figure~\ref{fig:disjoint1}(c), the chain connects to the other vertex on $\partial E_1$. However this creates a contradiction, since a curve following the arc of the boundary of $E_1$ through the green, then running across the tops of the bigons in the green, gives a simple closed curve that does not bound a disk in the green face. This contradicts the fact that each shaded face of $P_A$ is simply connected. Therefore the chain must connect with one of the vertices on $\partial E_2$. (Note this implies also that the faces that $\partial E_1$ and $\partial E_2$ meets must both be the same color, orange and green, as we have illustrated in the figures.) The same process may be repeated for the other two vertices. Then the faces between $\partial E_1$ and $\partial E_2$ consist of strings of adjacent bigons, hence $E_1$ is equivalent under parabolic compression to a subset of $E_s \cup E_2$.

Now suppose $\partial E_1$ and $\partial E_2$ intersect. In this case the region $(S\cup T)\setminus (S\cap T)$ consists of one or more connected components. Take one of these connected components, say $H$. Then $\partial H$ consists of an arc of $\partial E_1$ and an arc of $\partial E_2$, with two intersections of $\partial E_1$ and $\partial E_2$. Suppose these intersections both take place in the green face. Since the green face is simply-connected, we can push $E_1$ and $E_2$ to remove both of these intersections. 

\begin{figure}
\centering
\begin{tabular}{ccc}
\includegraphics{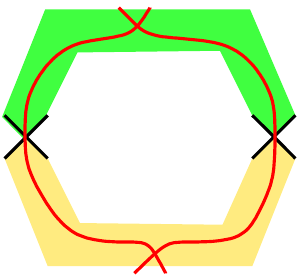} & 
\includegraphics{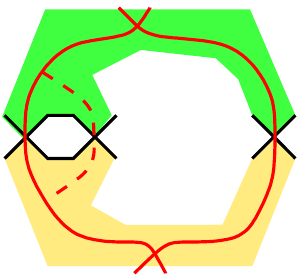} &
\includegraphics{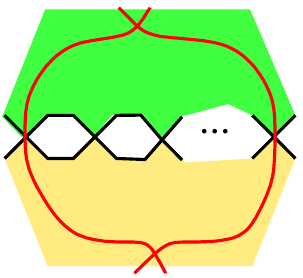} \\
(a) & (b) & (c) 
\end{tabular}
\caption{When $\partial E_1$ and $\partial E_2$ are not disjoint}
\label{fig:nondisjoint1}
\end{figure}

Thus we assume that the intersections of $\partial E_1$ and $\partial E_2$ occur in two different faces, green and orange as in Figure~\ref{fig:nondisjoint1}(a). Here the black lines represent edges of $P_A$. As before, by Lemma~\ref{lem:checkerboardidealpolyhedron}, the boundary of an EPD can only move from one colored face to another at an ideal vertex of $P_A$. Suppose $H$ is the region in the middle of this figure, so there are no non-bigon white faces in this region. The face just to the right of the vertex at the left of the figure is a white face, which must be a bigon, as in Figure~\ref{fig:nondisjoint1}(b). Also shown is a compression arc corresponding to a parabolic compression disk for $E_1$. Therefore $E_1$ is equivalent under parabolic compression to two new EPDs, one of which is simple. Then we can repeat this process, compressing off simple EPDs in a chain until the chain connects to the other vertex of $P_A$ (Figure~\ref{fig:nondisjoint1}(c)). Thus in this region, $E_1$ is equivalent under parabolic compression to a subset of $E_s \cup E_2$. 
Repeat the above arguments for each connected component of $\mathcal{H}$. We either push the intersections off, or compress bigons off of $E_1$ until $E_1$ is equivalent to a subset of $E_s \cup E_2$.
\end{proof}

Using Lemma~\ref{lem:equivalentunderparaboliccompression} we may more easily find a minimal set $E_c$ and therefore calculate $||E_c||$. For certain Montesinos links, $||E_c||$ is already known. The following is due to Futer, Kalfagianni, and Purcell \cite[Proposition~8.16]{FKP}.

\begin{thm}\label{thm:threepositive}
Let $D(K)$ be a reduced, admissible, non-alternating Montesinos diagram with at least three positive tangles. Then $||E_c|| = 0$.
\end{thm}

\subsection{Finding complex EPDs}
We now want to bound $||E_c||$ for classes of Montesinos links that do not fit under the umbrella of Theorem~\ref{thm:threepositive}. We will do so directly by finding a minimal spanning set of complex EPDs in $P_A$. The following results show how we might begin to look for complex EPDs.

\begin{lem}\label{lem:2edgeloops}
Let $D(K)$ be a reduced, admissible, $A$--adequate, non-alternating Montesinos diagram. Let $E$ be a complex EPD in $P_A$. Then either 
\begin{enumerate}
\item[(1)] $\partial E$ runs through a negative tangle $N_i$ of slope $-1 \leq q \leq -1/2$, along segments of $H_A$ that connect a single state circle, $I$, to the north and south sides of the tangle, as in Figure~\ref{fig:specnegtangle}, or
\item[(2)] there are exactly two positive tangles $P_1$ and $P_2$, and $\partial E$ runs along segments of $H_A$ that run through $P_1$ and $P_2$ from east to west.
\end{enumerate}
\end{lem}

\begin{proof}
By \cite[Corollary~6.6]{FKP}, $\partial E$ runs over tentacles adjacent to segments of a 2--edge loop in $\G'_A$.  Further, \cite[Lemma~8.14]{FKP} gives three possible forms of 2--edge loops in $G_A$. The first form of loop from that lemma corresponds to crossings in a single twist region, in which the all--$A$ resolution is the short resolution. By~\cite[Lemma~5.17]{FKP}, we may remove all bigons in the short resolution of a twist region without changing $||E_c||$. Therefore we may ignore such loops. For a 2--edge loop of the second form, \cite[Lemma~8.15]{FKP} gives that $\partial E$ must run through $I$. (Notice that the first paragraph in the proof of that lemma, proving this fact, does not use the hypothesis that $D(K)$ has at least three positive tangles.) Finally, a 2--edge loop of the third form runs over exactly two positive tangles.
\end{proof}

\begin{figure}
\centering
\includegraphics{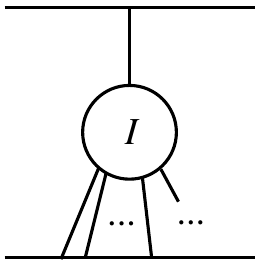}
\caption{A negative tangle with slope $-1 \leq q \leq -1/2$}
\label{fig:specnegtangle}
\end{figure}

Notice Lemma~\ref{lem:2edgeloops} also applies to an EPD that has been pulled into a normal square. This is because we can pull an EPD $E$ into a normal square without changing the segments that $\partial E$ runs along.

\begin{definition}\label{def:EPDtype}
A \emph{type~(1) EPD} is an EPD in $P_A$ described by Lemma~\ref{lem:2edgeloops}~(1); similarly for a type~(2) EPD.
\end{definition}

By Lemma~\ref{lem:2edgeloops}, every complex EPD in a Montesinos link must be either type~(1) or type~(2). This is not true for EPDs in general. Note also that an EPD may be both types~(1) and~(2).

The following lemma gives further information about type~(1) EPDs, and follows from the proof of \cite[Proposition~8.16]{FKP}.

\begin{lem} \label{lem:vertices}
Let $K$ be a $++-$ or a $+-+-$ Montesinos link, and let $E$ be a type~(1) complex EPD in $P_A$ as in Lemma~\ref{lem:2edgeloops}, with the innermost disk $I$ colored green, and $E$ pulled into a normal square as in Lemma~\ref{lemma:pullnormalsquare}.
Then the arcs of $\partial E$ running over a white face may occur only in the following places:
\begin{enumerate}
\item On the innermost disk $I$, either at the tail of a tentacle running to $I$ from the north or from the south, as in Figure~\ref{fig:type12vertices}(a) and (b).
\item At the head of a tentacle running out of the negative block containing $I$, if $\partial E$ runs from $I$ to an adjacent positive tangle, as in Figure~\ref{fig:type12vertices}(c) and (d). 
\item On the next adjacent negative block, if $\partial E$ runs downstream from $I$ to an adjacent positive tangle, then across a segment spanning the positive tangle east to west, and then downstream across the outer state circle of the next negative block. See Figure~\ref{fig:type3vertices}. This situation cannot occur for $++-$ links.
\end{enumerate}
\end{lem}

\begin{figure}
  \centering
\begin{tabular}{cccc}
  \includegraphics{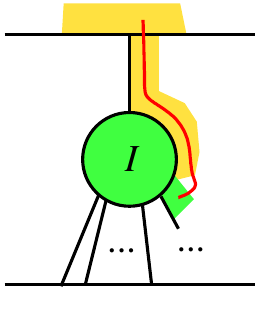} &
  \includegraphics{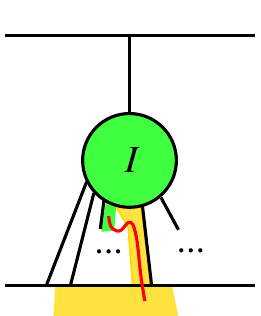} &
  \includegraphics{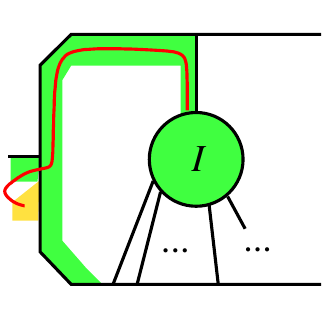} &
  \includegraphics{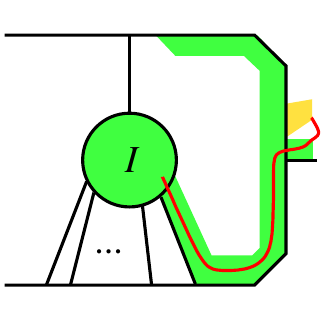} \\
  (a) North & (b) South & (c) North & (d) South
  \end{tabular}
  \caption{Type (1N) and (1S), and type (2N) and (2S) white edges}
  \label{fig:type12vertices}
\end{figure}

\begin{figure}
\centering
\begin{tabular}{ccc}
  \includegraphics{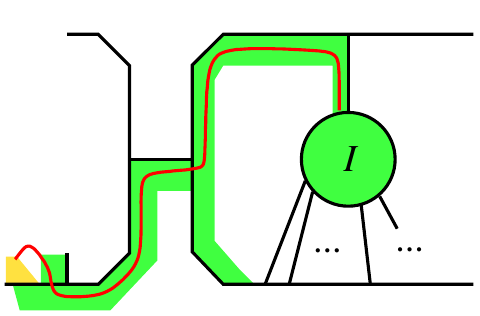} & \hspace{.1in} &
  \includegraphics{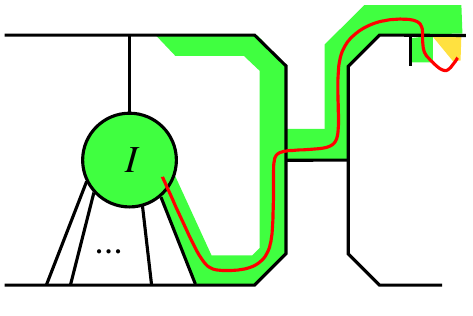} \\
  (a) Type 3 North & & (b) Type 3 South
\end{tabular}
  \caption{Type (3N) and (3S) white edges}
  \label{fig:type3vertices}
\end{figure}

We will call these white arcs of types (1N), (1S), (2N), (2S), (3N), and (3S). Lemma~\ref{lem:vertices} requires that $I$ be colored green. However, recall that given an EPD, we may choose how to label faces green and gold. Hence classifying all type~(1) EPDs will involve only considering the combinations resulting from the lemma.

\begin{proof}
In \cite[Proposition~8.16]{FKP}. In that proof, there are five types of arcs in white faces. However, types~(4) and~(5) require non-prime arcs, and we do not have any non-prime arcs in our diagrams. Notice that for a type~(3) arc, we must have a string of tangles $-+-$ moving either east or west, where the first negative tangle is the one containing $I$. This string of tangles cannot occur in a Montesinos link with two positive and one negative tangle. Thus $++-$ links may only have types~(1) and~(2) white edges.
\end{proof}

We will often obtain information about a complex EPD $E$ by the following argument. Suppose $E$ is pulled into a normal square with $\partial E$ running through a given white face, cutting off a triangle. For any given white side, the two shaded faces met by $\partial E$ can originate only at a few locations, which can be determined by considering the forms of the polyhedra as in Lemmas~\ref{lemma:Pdecomp-ppm} and~\ref{lemma:Pdecomp-pmpm}, and the $A$--resolutions of the tangles. By considering locations where the faces originate, we can identify $\partial E$. The following lemma is also useful. 

\begin{lem}[Two--Face Argument]
  \label{lemma:twoface}
Let $E$ be a normal square such that $\partial E$ meets shaded faces $F_1$ and $F_2$ at one arc in a white face $W_1$, and meets shaded faces $F_1$ and $F_3$ in another arc in a white face $W_2$. Then $F_2=F_3$. Moreover, up to isotopy there is a unique embedded arc from $W_1$ to $W_2$ running through $F_2=F_3$, and $\partial E$ must follow this arc.
\end{lem}

\begin{proof}
The fact that $F_2=F_3$ is obvious, since a normal square only runs through two shaded faces. The fact that there is a unique arc from $W_1$ to $W_2$ in $F_2$ follows from the fact that shaded faces are simply connected. 
\end{proof}


\section{Bounds on essential product disks}\label{sec:epds}
The main results of this section are Propositions~\ref{prop:++-prop} and~\ref{prop:+-+-prop}, which bound $||E_c||$ for $++-$ and $+-+-$ Montesinos links, respectively. We prove these by looking first for complex EPDs that are of type~(1), then for those that are not of type~(1), which must necessarily be type~(2). 

\subsection{The $++-$ case}

\begin{lem}Let $K$ be a $++-$ link with reduced, admissible diagram $D(K)$. Then there are only two type~(1) complex EPDs in $P_A$, one with white faces of type (1N) and (1S) as shown in Figure~\ref{fig:++-1N1S}(b), and the other with white faces of type (2N) and (1S) as shown in Figure~\ref{fig:++-2N1S}(c).
\label{lem:++-type1EPDs}
\end{lem}

\begin{proof}
Let $E$ be a type~(1) complex EPD in $P_A$. Pull $E$ into a normal square with two white edges as in Lemma~\ref{lemma:pullnormalsquare}. Lemma~\ref{lem:vertices} tells us all the possible locations for white edges of $\partial E$. These possibilities are labeled (1N), (1S), (2N), and (2S). We must have one white edge north and one south. Therefore we must check only the following four combinations.

\underline{White edges (1N), (1S).}
By the notation (1N), (1S), we mean that our two white edges are described by type (1N) and type (1S) in Lemma~\ref{lem:vertices}. In this case, our EPD must have a portion as in Figure~\ref{fig:++-1N1S}(a). By the Two--Face Argument (Lemma~\ref{lemma:twoface}), the faces to the north and south, colored gold and magenta, must agree. Therefore, they must originate in the same innermost disk. We now consider where these faces might originate.

The face from the south, labeled magenta, must run around the bottom of the diagram on the inside, and therefore originate in the positive tangle $T_a$, or in the state circle $G$ between $T_a$ and $T_b$ if $T_a$ contains no state circles. The face from the north, labeled gold, runs from $T_b$, and therefore originates in $G$. Thus both originate in $G$ and $T_a$ contains no state circles. This is shown in Figure~\ref{fig:++-1N1S}(b). Here $\partial E$ may bound white faces that are not bigons on both sides, where indicated by an asterisk in the figure. Thus this may be a complex EPD.

\begin{figure}
\centering
\begin{tabular}{cc}
	\includegraphics{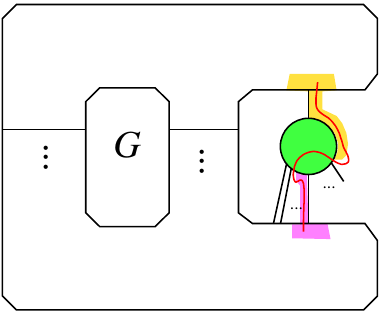} &
  \includegraphics{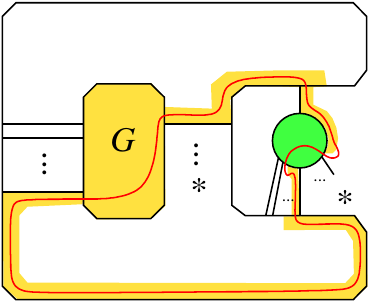} \\
  (a) & (b) \\
\end{tabular}
\caption{An EPD with white edges type (1N) and (1S)}
\label{fig:++-1N1S}
\end{figure}

\underline{White edges (1N), (2S).}
The boundary of an EPD with white edges type (1N) and (2S) will be as in Figure~\ref{fig:++-1N2S}(a), with the EPD following a green tentacle wrapped around the south of the diagram on the outside. A white edge of type (2S) requires that $T_c$ has only one state circle, so that there are no non-bigons in $T_c$.

Now, the gold tentacle meeting $I$ from the north originates in $G$. The gold tentacle meeting the green on the far west of the diagram will originate in $T_a$ or in $G$; since the gold only has one innermost disk from which it originates, this must be $G$. From $G$, the gold face reaches the type $(1N)$ white edge by running across $T_b$; but it may reach the $(2S)$ white edge by running across either $T_a$ or $T_b$.  Therefore there are two possible EPDs, shown in Figures~\ref{fig:++-1N2S}(b) and~\ref{fig:++-1N2S}(c). In each of these figures the only places where non-bigons may occur in the diagram are indicated by an asterisk. Note each of the EPDs bounds only bigons on one side, so they are both semi-simple. 

\begin{figure}
\centering
\begin{tabular}{ccc}
	\includegraphics{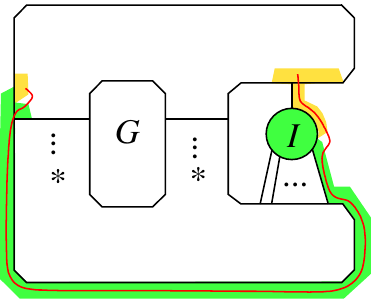} &
	\includegraphics{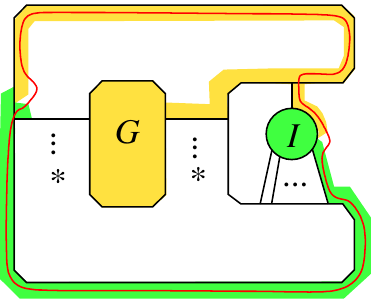} &
	\includegraphics{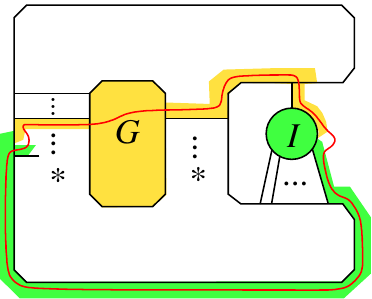} \\
  (a) & (b) & (c)
\end{tabular}
\caption{Semi-simple EPDs with white edges of type (1N) and (2S)}
\label{fig:++-1N2S}
\end{figure}

\underline{White edges (2N), (1S).} 
A normal square with white edges (2N) and (1S) will have a diagram as in Figure~\ref{fig:++-2N1S}(a). The gold face to the south wraps around the bottom of the diagram and originates in $T_a$ or in $G$. The magenta face to the west, inside $T_b$, must either originate in $T_b$, or in $G$, or it is the same tentacle as the gold to the south.

Suppose first that the magenta face to the west is the same tentacle as the gold to the south of $T_c$. Then as shown in Figure~\ref{fig:++-2N1S}(b), the EPD follows this tentacle from gold to the south to the tip of the face to the west, and the EPD is semisimple, encircling only bigons of $T_c$.

So suppose these are not the same tentacle. Since they lie in the same face, by Lemma~\ref{lemma:twoface}, both must originate in $G$, as in Figure~\ref{fig:++-2N1S}(c). Thus $T_a$ has only bigon faces, but $T_b$ and $T_c$ might have non-bigon faces. The location of non-bigon faces is indicated by asterisks. Note that this may give a complex EPD.

\begin{figure}
\centering
\begin{tabular}{ccc}
  \includegraphics{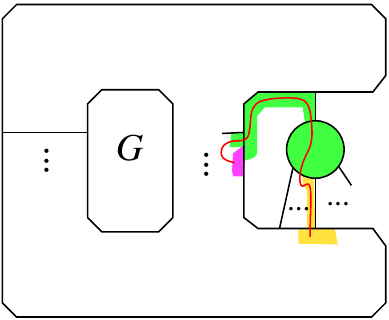} &
 	\includegraphics{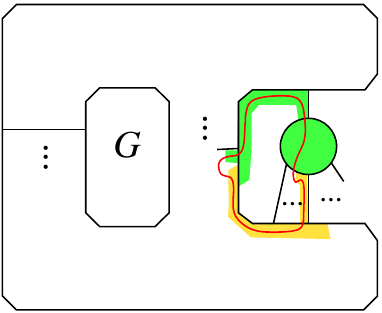} &
	\includegraphics{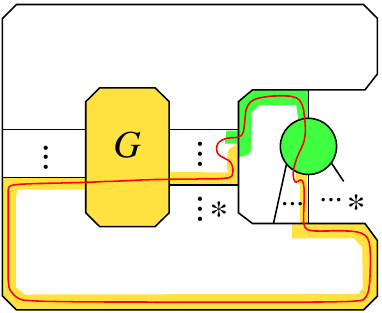} \\
  (a) & (b) & (c)
\end{tabular}
\caption{EPDs with white edges of type (2N) and (1S)}
\label{fig:++-2N1S}
\end{figure}

\underline{White edges (2N), (2S).}
The (2N) and (2S) white edges give a diagram as in Figure~\ref{fig:++-2N2S}(a). Note the gold and magenta faces must actually be the same face. We determine where gold and magenta tentacles may originate. This case is somewhat more complicated than the previous three. The magenta may originate in $T_b$ or in $G$, or the magenta tentacle may agree with the tentacle running around the south of the diagram, and therefore originate at the furthest southwest state circle of $T_a$, or in $G$ if $T_a$ has no state circles. The gold may originate in $G$ and run east to wrap around the north of the diagram, may originate in $G$ and run west across a single segment in $T_a$, or may originate in $T_a$. Putting this information together, noting that gold and magenta originate in the same place (Lemma~\ref{lemma:twoface}), we find that the face either originates in $G$, or in the far southwest state circle in $T_a$.

If the face originates in $G$ there are four different possible diagrams, depending on whether gold or magenta run east or west, shown in Figures~\ref{fig:++-2N2S}(b) through (e).  If the face originates in $T_a$ the diagram is as in Figure~\ref{fig:++-2N2S}(f). Note that all of the EPDs pictured bound only bigons on one side, and therefore are semi-simple.

\begin{figure}
\begin{center}
\begin{tabular}{ccc}
\includegraphics{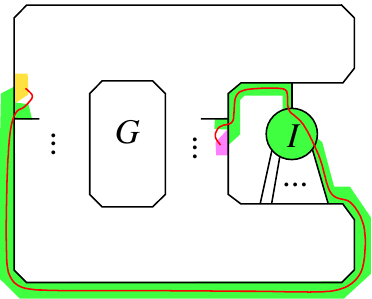} &
\includegraphics{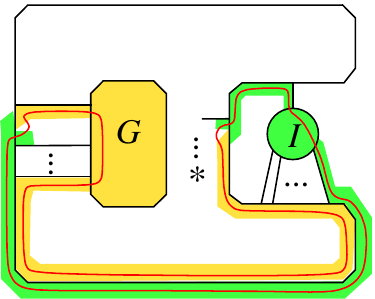} &
\includegraphics{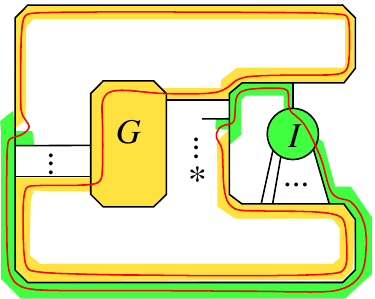} \\
(a) & (b) & (c) \\
\includegraphics{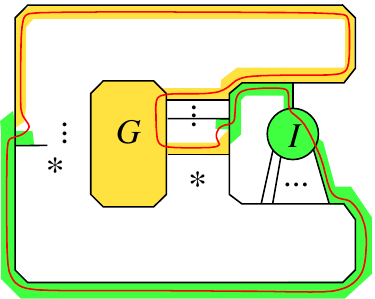} &
\includegraphics{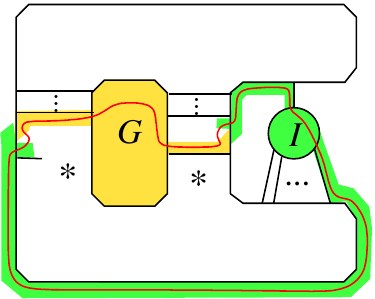} &
\includegraphics{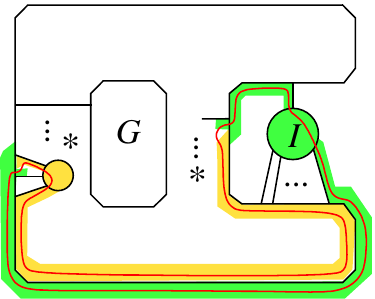}
\\
(d) & (e) & (f)
\end{tabular}
\end{center}
\caption{EPDs with white edges of type (2N), (2S)}
\label{fig:++-2N2S}
\end{figure}

This concludes the search for type (1) complex EPDs in $P_A$. The complex EPDs are found are depicted in Figures~\ref{fig:++-1N1S}(b) and~\ref{fig:++-2N1S}(c). 
\end{proof}

\begin{lem} Let $K$ be a $++-$ link with reduced, admissible diagram $D(K)$. Then the only possible complex EPDs in $P_A$ which are not of type~(1) are those shown in Figure~\ref{fig:aSbN++-Comb}(c) and~(d).
\label{lem:++-type2EPDs}
\end{lem}

\begin{proof}
Let $E$ be a complex EPD in $P_A$ that is not type~(1); then $E$ must be type~(2). Therefore, $\partial E$ runs along segments across both $T_a$ and $T_b$. Each such segment has an endpoint on the innermost disk $G$. Since we may choose how to label the colors an EPD runs through, we choose to color the face bounded by $G$ gold. 
By Lemma~\ref{lemma:pullnormalsquare}, we may pull $E$  into a normal square with white sides at tails of the gold tentacles, and heads of the green.

For convenience, call the segment on the west $\alpha$, and the one on the east $\beta$. A priori, $\partial E$ may run along either the north or south sides of these segments. However, in fact by labeling $G$ gold, we may restrict to the case that $\partial E$ runs along the south of $\alpha$ and the north of $\beta$. This is because white sides occur at tails of gold and heads of green. There are no heads of any tentacles adjacent $G$, except those which run from innermost disk $G$ itself, which are gold, so there can be no white sides adjacent $G$. Thus $\partial E$ must run away from $G$ through gold tentacles, which run south of $\alpha$ and north of $\beta$.

Tentacles running along the south side of $\alpha$ and the north side of $\beta$ both originate in $G$. To find $\partial E$, we must look for a white side downstream from both $\alpha$ and $\beta$, and these must occur at a tail of a gold tentacle.

We have three possibilities for the location of the white side downstream from $\alpha$:
\begin{enumerate}
\item[($\alpha$1)] The tentacle on the south side of $\alpha$ terminates in $T_a$. In this case the white side lies in $T_a$, and meets the head of a green tentacle coming from a tentacle originating in the far southeast of $T_c$. See Figure~\ref{fig:++-aSbNalpha}($\alpha$1).
\item[($\alpha$2)] The tentacle runs out of $T_a$, across the entire south of the diagram, and $\partial E$ follows it to where it terminates in $T_b$. The white side lies in $T_b$, and meets a green tentacle originating in the far northwest of $T_c$. See Figure~\ref{fig:++-aSbNalpha}($\alpha$2).
\item[($\alpha$3)] The tentacle runs across the south of the diagram, meeting the head of a new gold tentacle in $T_c$, and $\partial E$ runs into $T_c$. Such a tentacle must terminate on a state circle in $T_c$, and so the white side is in $T_c$, meeting a green tentacle in $T_c$ originating at the state circle on which the gold tentacle terminates. See Figure~\ref{fig:++-aSbNalpha}($\alpha$3). 
\end{enumerate}

\begin{figure}
  \begin{center}
  \begin{tabular}{ccc}
    \includegraphics{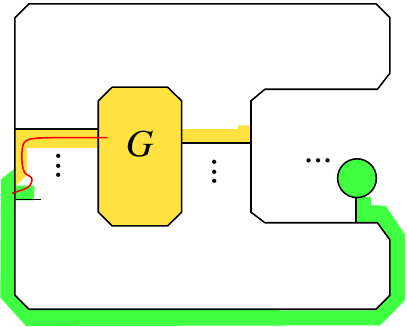} &
    \includegraphics{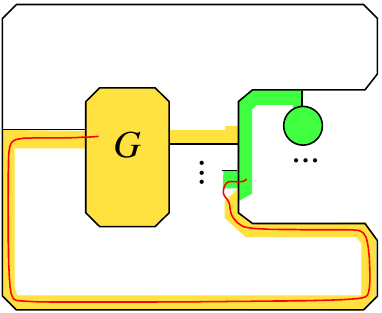} &
    \includegraphics{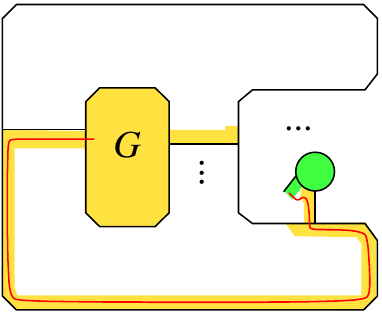} \\
    ($\alpha$1) & ($\alpha$2) & ($\alpha$3)
  \end{tabular}
  \end{center}
  \caption{$\partial E$ runs through a gold face across the south of $\alpha$; shown are possible locations of white sides.}
  \label{fig:++-aSbNalpha}
\end{figure}

Similarly, there are three possibilities for the white side downstream from $\beta$:
\begin{enumerate}
\item[($\beta$1)] The tentacle across the north of $\beta$ terminates in $T_b$, and the white side lies in $T_b$. It meets a green tentacle originating at the state circle in the far northwest of $T_c$. See Figure~\ref{fig:++-aSbNbeta}($\beta$1).
\item[($\beta$2)] The tentacle across the north of $\beta$ runs all the way around the north of the diagram, terminating in $T_a$, and $\partial E$ follows this tentacle to $T_a$. The white side lies in $T_a$, and meets a green face with head on the tentacle running across the bottom of the diagram, originating in the far southeast state circle of $T_c$. See Figure~\ref{fig:++-aSbNbeta}($\beta$2).
\item[($\beta$3)] The tentacle across the north of $\beta$ runs around the north of the diagram, but $\partial E$ follows a new tentacle into $T_c$. This tentacle terminates on a state circle in $T_c$, and the white side must be in $T_c$, connecting to a green face that originates in this state circle. See Figure~\ref{fig:++-aSbNbeta}($\beta$3).
\end{enumerate}

\begin{figure}
  \begin{center}
    \begin{tabular}{ccc}
    \includegraphics{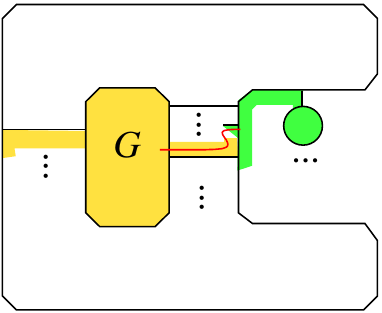} &
    \includegraphics{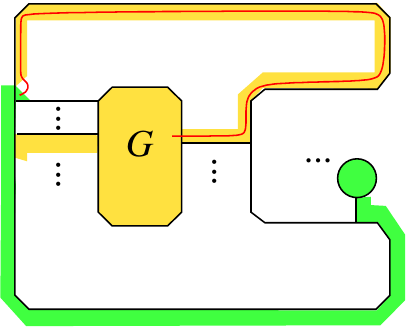} &
    \includegraphics{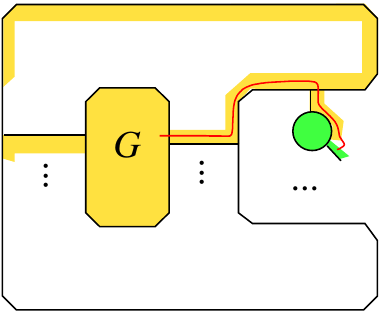}\\
    ($\beta$1) & ($\beta$2) & ($\beta$3)
    \end{tabular}
  \end{center}
  \caption{$\partial E$ runs through a gold face across the north of $\beta$; shown are possible locations of white edges and origins of the green tentacles they meet}
  \label{fig:++-aSbNbeta}
\end{figure}

The above give nine possible combinations, which we reduce. Note that in order for ($\alpha$1) to combine with ($\beta$1) or ($\beta$3), there can only be one state circle in $T_c$ and $\partial E$ must run over it. This means $E$ is a type~(1) EPD, contradicting assumption. An identical argument rules out ($\alpha$2) combined with ($\beta$2), and ($\alpha$3) combined with either ($\beta$1) or ($\beta$3).

We work through the remaining combinations. First, ($\alpha$1) and ($\beta$2), shown in Figure~\ref{fig:aSbN++-Comb}(a), gives rise to a semi-simple EPD. Similarly, ($\alpha$3) and ($\beta$2), shown in Figure~\ref{fig:aSbN++-Comb}(b), gives a semi-simple EPD. 
The combination ($\alpha$2) combined with ($\beta$1) may give a complex EPD, as does ($\alpha$2) combined with ($\beta$3). These are shown in Figure~\ref{fig:aSbN++-Comb}(c) and (d).
\end{proof}

\begin{figure}
  \begin{center}
  \begin{tabular}{cccc}
    \includegraphics[width=1.2in]{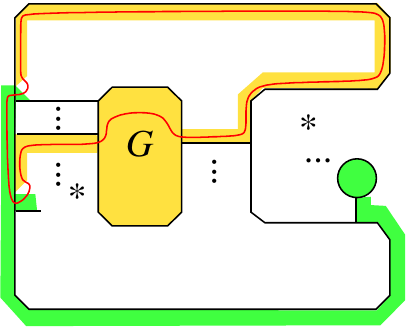} &
    \includegraphics[width=1.2in]{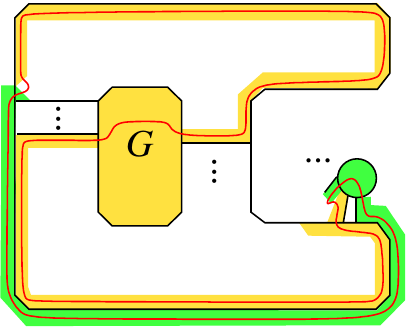} &
    \includegraphics[width=1.2in]{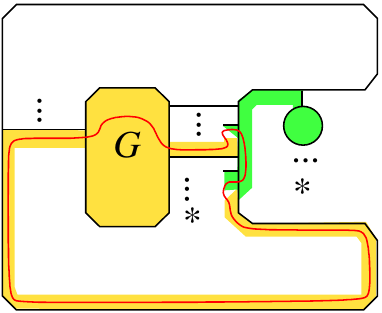} &
    \includegraphics[width=1.2in]{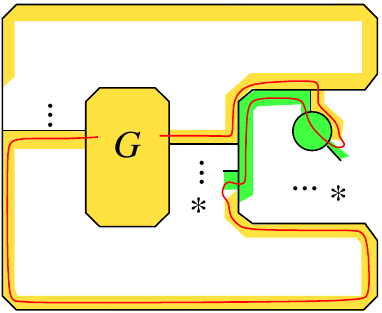} \\
    (a) ($\alpha$1) ($\beta$2) & (b) ($\alpha$3) ($\beta$2) &
    (c) ($\alpha$2) ($\beta$1) &
    (d) ($\alpha$2) ($\beta$3) \\
  \end{tabular}
  \end{center}
  \caption{Possibilities for $\partial E$ running over the south of $\alpha$ and the north of $\beta$}
  \label{fig:aSbN++-Comb}
\end{figure}

\begin{prop}
Let $K$ be a $++-$ link with reduced, admissible diagram $D(K)$. Then
\(
 ||E_c|| \leq 1,
\)
where $||E_c||$ is the number of complex EPDs required to span $P_A$.
\label{prop:++-prop}
\end{prop}

\begin{proof}
To obtain the desired result, we need to show that if $E$ and $E'$ are two complex EPDs in $P_A$, then $E'$ is equivalent under parabolic compression to a subset of $E\cup E_s$. 

In Lemmas~\ref{lem:++-type1EPDs} and~\ref{lem:++-type2EPDs}, we found all possible complex EPDs in $P_A$. They are shown in Figures~\ref{fig:++-1N1S}(b), \ref{fig:++-2N1S}(c), \ref{fig:aSbN++-Comb}(c), and \ref{fig:aSbN++-Comb}(d), which we reproduce for convenience in Figure~\ref{fig:++-complex}. Denote the four complex EPDs by $E_1$, $E_2$, $E_3$, and $E_4$, as illustrated. We compare these pairwise.

\begin{figure}
  \centering
  \begin{tabular}{cccc}
    \includegraphics[width=1.2in]{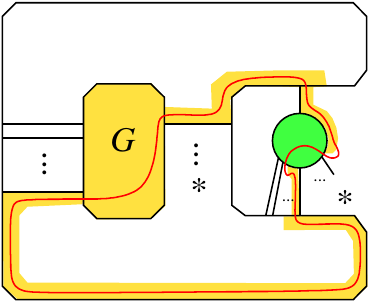} &
    \includegraphics[width=1.2in]{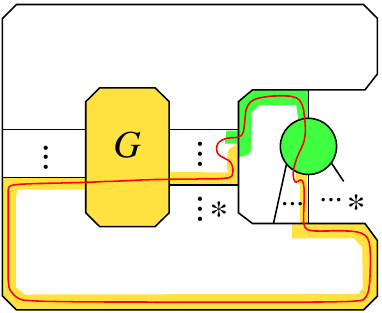} &
    \includegraphics[width=1.2in]{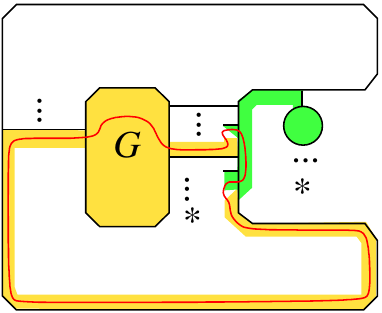} &
    \includegraphics[width=1.2in]{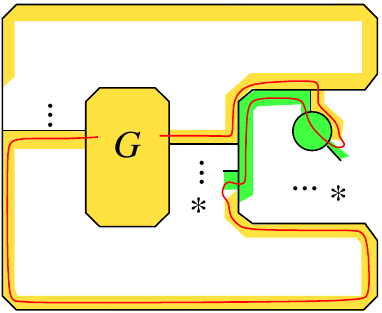} \\
    $E_1$ & $E_2$ & $E_3$ & $E_4$
  \end{tabular}
  \caption{All complex EPDs in $P_A$ for a $++-$ link}
  \label{fig:++-complex}
\end{figure} 

First consider $E_1$ and $E_2$. Note both EPDs run over the southernmost segment of $T_a$. In $T_b$, $E_1$ runs over the northernmost segment, and $E_2$ runs over a segment besides the northernmost one that still spans $T_b$. Note they both run from the segment spanning $T_b$ to the north of $T_c$, and so there are only bigons in $T_b$ between them. In $T_c$, then run from north to south, and either one may run down any segments connecting the state circle $I$ to the state circle on the south side of $T_c$. From there, both run along the inside south face, back to the southernmost segment of $T_a$. Thus there are only bigons in $T_c$ and in $T_b$ between the two EPDs. Therefore by Lemma~\ref{lem:equivalentunderparaboliccompression}, $E_1$ and $E_2$ are equivalent under parabolic compression.

The analysis of other pairs is similar. For $E_1$ and $E_3$, the two EPDs differ by only bigons in $T_b$, and bigons in $T_c$, including the bigon on the far west side of $T_c$. Similarly for $E_1$ and $E_4$. The pairs $E_2$ and $E_3$ differ only by bigons in $T_b$ and in $T_c$, as do the pairs $E_2$ and $E_4$. Finally, $E_3$ and $E_4$ differ only by bigons in $T_b$. So in all cases, Lemma~\ref{lem:equivalentunderparaboliccompression} implies that the EPDs are equivalent under parabolic compression. 
\end{proof}

\subsection{The $+-+-$ case}

\begin{lem}
Let $K$ be a $+-+-$ link with reduced, admissible diagram $D(K)$. Then a type~(1) complex EPD in $P_A$ is equivalent to one with white edges of type (1N) and (1S), shown in Figure~\ref{fig:+-+-1N1S}(b), or of type (2N) and (1S), shown in Figure~\ref{fig:+-+-2N1S}(c), or of type (3N) and (2S), shown in Figure~\ref{fig:+-+-3N2S}(b).
\label{lem:+-+-type1EPDs}
\end{lem}

\begin{proof}
For $E$ a type~(1) complex EPD in $P_A$, we pull $E$ into a normal square with two white edges (Lemma~\ref{lemma:pullnormalsquare}). Since $E$ is type~(1), $\partial E$ runs through a negative tangle from north to south. We may assume, after possibly taking a cyclic permutation, that $\partial E$ runs through $T_d$. Lemma~\ref{lem:vertices} tells us all the possible locations for white edges of $\partial E$. These possibilities are labeled (1N), (1S), (2N), (2S), (3N), and (3S). Since $\partial E$ runs over the segment to the north and a segment to the south of the state circle $I$ in $T_d$, there are nine combinations to check.

\underline{White edges of type (1N), (1S).}
By the notation (1N), (1S), we mean that our two white edges are described by type~(1N) and type~(1S) in Lemma~\ref{lem:vertices}. Our EPD must have white edges as in Figure~\ref{fig:+-+-1N1S}(a). By Lemma~\ref{lemma:twoface}, the gold and magenta faces shown must agree. The magenta tentacle has its head on the tentacle that wraps all the way around the bottom of the diagram. This tentacle either originates in the far southwest state circle of $T_a$, or if $T_a$ has no state circles, from the far northwest state circle of $T_b$. The gold tentacle has its head on the tentacle that wraps around the inside of the top of the diagram; it must run over a segment connecting east to west of $T_c$, and originates in the far southwest state circle of $T_b$. For these originating state circles to be the same, there must be only one state circle in $T_b$, and gold and magenta must originate there. Thus there are no state circles in $T_a$. This results in the diagram pictured in Figure~\ref{fig:+-+-1N1S}(b). This diagram may contain non-bigon white faces where indicated by an asterisk, so this EPD may be complex.

\begin{figure}
  \begin{center}
    \begin{tabular}{cc}
      \includegraphics{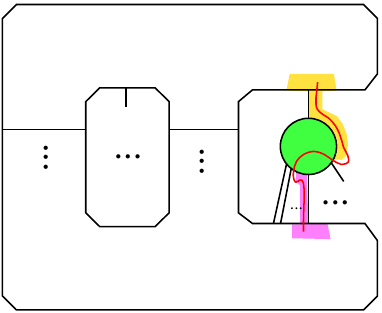} &
      \includegraphics{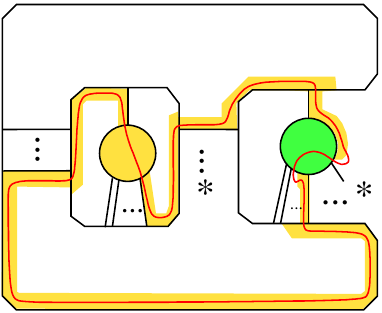} \\
      (a) & (b) \\
    \end{tabular}
  \end{center}
  \caption{For $+-+-$ link, edges of type (1N) and (1S); possible complex EPD shown in (b)}
  \label{fig:+-+-1N1S}
\end{figure}

\underline{White edges (1N), (2S).}
In this case, $\partial E$ must run over faces of $P_A$ that look like Figure~\ref{fig:+-+-1N2S}(a). Since a white edge is of type~(2S), we only have one state circle in $T_d$. The gold tentacle in this figure has its head on the tentacle running across the inside--top of the diagram; this originates in the furthest southeast state circle of $T_b$. By Lemma~\ref{lemma:twoface}, the magenta tentacle must originate from the same state circle. However, the magenta tentacle either originates in $T_a$, or runs through $T_a$ and originates in the northwest of $T_b$, or runs across the inside--top of the diagram and originates in the southeast of $T_b$. Originating in $T_a$ is impossible, but the other two options are both valid, provided in the case the face originates in the northwest of $T_b$ there is only one state circle in $T_b$. These two valid options result in Figures~\ref{fig:+-+-1N2S}(b) and~(c), respectively. Note the EPD in Figure~\ref{fig:+-+-1N2S}(c) bounds a single bigon on the outside, and so is simple. In Figure~\ref{fig:+-+-1N2S}(b), non-bigons can only occur where indicated by an asterisk. This EPD is semi-simple.

\begin{figure}
  \begin{center}
    \begin{tabular}{ccc}
      \includegraphics{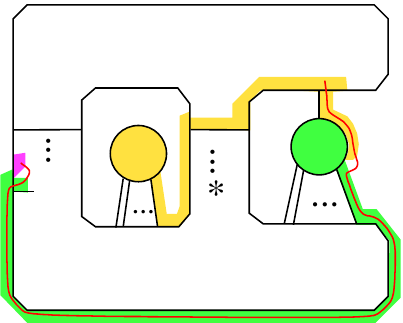} &
      \includegraphics{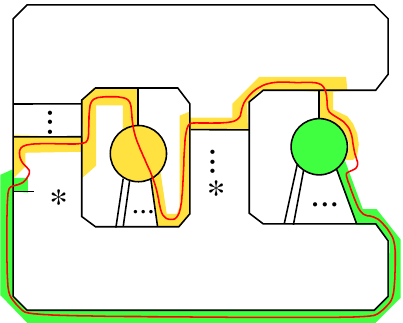} &
      \includegraphics{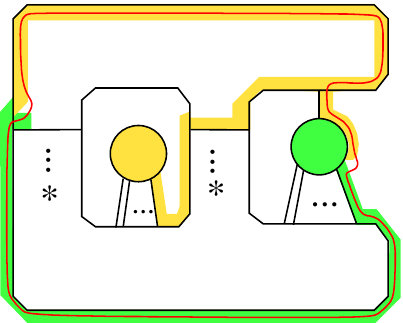} \\
      (a) & (b) & (c)
    \end{tabular}
  \end{center}
  \caption{White edges of type (1N) and (2S)}
  \label{fig:+-+-1N2S}
\end{figure}

\underline{White edges (1N), (3S).}
White edges of type (1N) and (3S) produce a diagram as in Figure~\ref{fig:+-+-1N3S}. Because a white edge is of type (3S), $T_d$ must have only one state circle. In this case, both gold tentacles originate in the southeast corner of the tentacle $T_b$, as shown, and only one diagram is possible. White faces that are not bigons may occur only where indicated by an asterisk. Notice that the outside of $\partial E$ in this case bounds only bigons, so the EPD is semi-simple.

\begin{figure}
  \centering
  \includegraphics{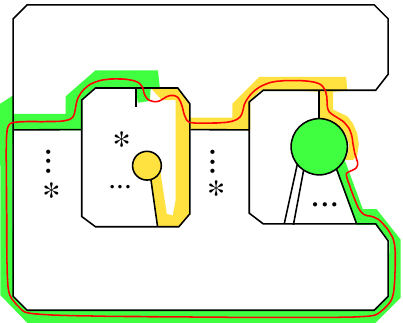}
  \caption{White edges of type (1N) and (3S)}
  \label{fig:+-+-1N3S}
\end{figure}

\underline{White edges (2N), (1S).}
A normal square with white edges of type (2N) and (1S) will have a diagram as in Figure~\ref{fig:+-+-2N1S}(a). The gold tentacle has its head on the tentacle wrapping all along the inside--bottom of the diagram.
Thus the gold face originates in the far southwest state circle in $T_a$, or if $T_a$ has no state circles, in the far northwest state circle of $T_b$. 
The magenta tentacle either originates in $T_c$, or runs from east to west in $T_c$ and originates in the southeast corner of $T_b$, or it might be the same tentacle that runs across the inside--bottom of the diagram.

\begin{figure}
  \centering
  \begin{tabular}{ccc}
    \includegraphics{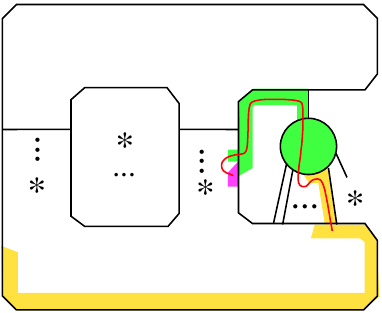} &
    \includegraphics{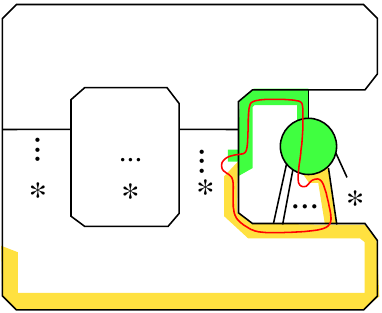} &
    \includegraphics{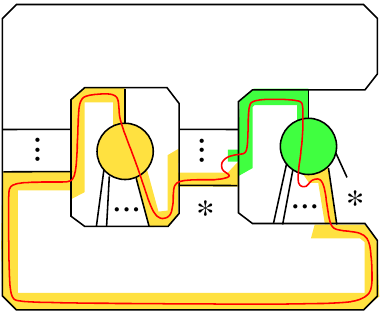} \\
    (a) & (b) & (c)
  \end{tabular}
  \caption{Type (2N) and (1S)}
  \label{fig:+-+-2N1S}
\end{figure}

If the magenta tentacle is the tentacle that runs across the inside--bottom of the diagram, then we can close up $\partial E$ in such a way that it encircles only bigons in $T_d$, as in Figure~\ref{fig:+-+-2N1S}(b). Because the gold face is simply connected, this is the only way to connect $\partial E$ up to homotopy, and so $E$ is semi-simple in this case. 

Since gold and magenta must originate in the same state circle, the only remaining possibility is that both originate in $T_b$, so $T_a$ contains no state circles, and in $T_b$, the northwest and southeast state circles agree, so $T_b$ has only one state circle. The result is shown in Figure~\ref{fig:+-+-2N1S}(c). Non-bigons may be present where indicated by an asterisk; the EPD may be complex.

\underline{White edges (2N), (2S).}
A normal square with white edges of type (2N) and (2S) has diagram as in Figure~\ref{fig:+-+-2N2S}(a). Notice that type (2S) forces $T_d$ to have only one state circle. The magenta tentacle from (2S) either
\begin{itemize}
\item originates in $T_a$,
\item or runs across $T_a$ and originates in the northwest of $T_b$,
\item or comes from the tentacle running across the inside north of the diagram, and therefore originates in the southeast of $T_b$.
\end{itemize}
The gold tentacle from (2N) either
\begin{itemize}
\item originates in $T_c$,
\item or runs across $T_c$ and originates in the southeast of $T_b$,
\item or comes from the tentacle running across the inside south of the diagram, and therefore originates the south of $T_a$,
\item or if $T_a$ has no state circles, comes from the tentacle running across the inside south of the diagram and originates in the northwest of $T_b$.
\end{itemize}

\begin{figure}
  \centering
  \begin{tabular}{ccc}
    \includegraphics{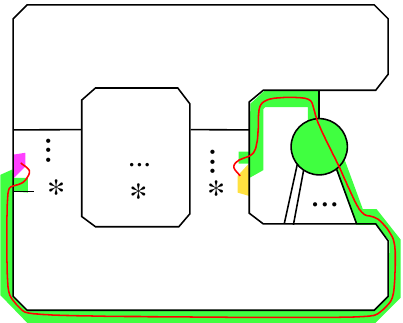} &
    \includegraphics{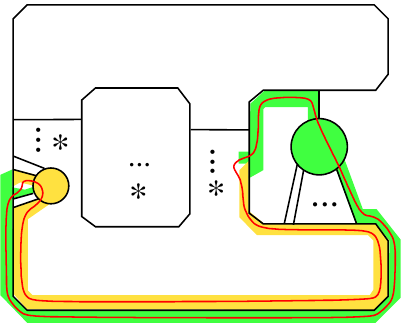} &
    \includegraphics{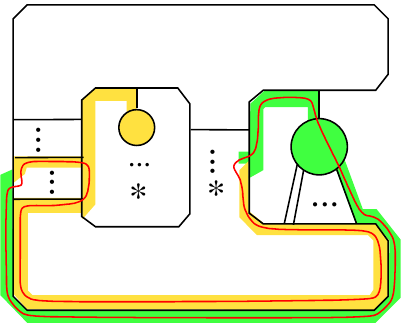} \\
    (a) & (b) & (c) \\
    \includegraphics{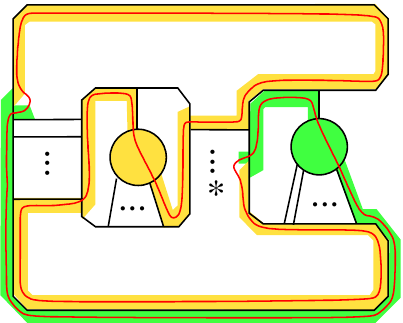} &
    \includegraphics{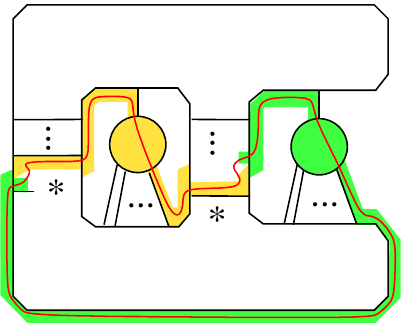} &
    \includegraphics{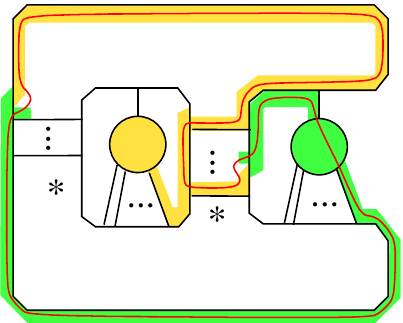} \\
    (d) & (e) & (f)
  \end{tabular}
  \caption{Type (2N) and (2S)}
  \label{fig:+-+-2N2S}
\end{figure}

Since magenta and gold originate in the same place, (2N) cannot originate in $T_c$, but there are five possibilities remaining, as follows. First, gold originates in $T_a$; this is shown in Figure~\ref{fig:+-+-2N2S}(b). Note it is semi-simple. Second, gold originates in $T_b$, the tentacle from (2N) runs around the inside south of the diagram, and the tentacle from (2S) originates in the northwest of $T_b$; this is shown in Figure~\ref{fig:+-+-2N2S}(c), and again it is semi-simple. Third, gold originates in $T_b$, the tentacle from (2N) runs around the inside south of the diagram, and the tentacle from (2S) runs around the inside north of the diagram, as in Figure~\ref{fig:+-+-2N2S}(d). Note this is also semi-simple. Fourth, gold originates in $T_b$, the tentacle from (2N) runs west to east across $T_c$, the tentacle from (2S) runs east to west across $T_a$, as in Figure~\ref{fig:+-+-2N2S}(e), which is semi-simple. Finally fifth, gold originates in $T_b$, the tentacle from (2N) runs west to east across $T_c$, and the tentacle from (2S) runs across the inside north of the diagram, as in Figure~\ref{fig:+-+-2N2S}(f). Again this is semi-simple, bounding only bigons to the outside. 

\underline{White edges (2N), (3S).}
White edges (2N) and (3S) give the diagram in Figure~\ref{fig:+-+-2N3S}(a). Since we have a (3S) white edge, there can be only one state circle in $T_d$.
The magenta tentacle originates in the southeast state circle in $T_b$. Therefore, the gold tentacle must also originate in $T_b$. This can happen one of two ways: either the gold tentacle runs west to east across $T_c$, originating in the southeast of $T_b$, or the gold tentacle runs all across the inside south of the diagram, across a tentacle of $T_a$ (which must have no state circles), and originates in the northeast of $T_b$ (which must have exactly one state circle). These two options are shown in Figure~\ref{fig:+-+-2N3S}(b) and~(c). Note both bound only bigons to one side, hence both are semi-simple.

\begin{figure}
  \centering
  \begin{tabular}{ccc}
    \includegraphics{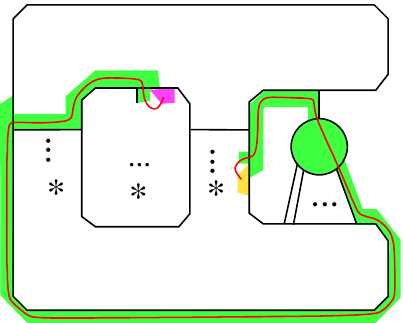} &
    \includegraphics{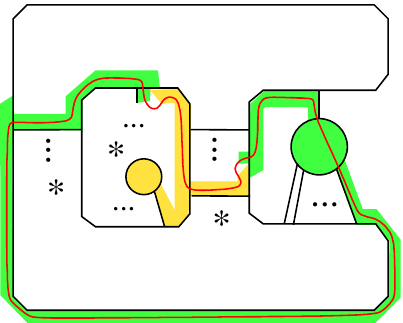} &
    \includegraphics{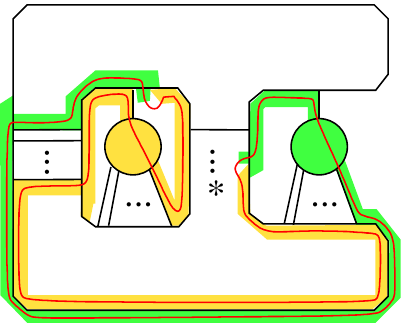} \\
    (a) & (b) & (c)
  \end{tabular}
  \caption{Type (2N) and (3S)}
  \label{fig:+-+-2N3S}
\end{figure}

\underline{White edges (3N), (1S).}
The diagram appears as in Figure~\ref{fig:+-+-3N1S}(a). Note that a type~(3N) white edge requires that $T_c$ has no state circles. The magenta tentacle originates in $T_b$, hence the gold also must originate in $T_b$. This means $T_a$ has no state circles, and the gold tentacle runs across $T_a$, originating in the northwest of $T_b$. Then the magenta originates in the same circle, and the diagram is as in Figure~\ref{fig:+-+-3N1S}(b). Note this gives a semi-simple EPD.

\begin{figure}
  \centering
  \begin{tabular}{cc}
    \includegraphics{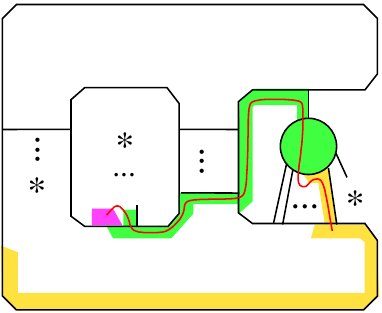} &
    \includegraphics{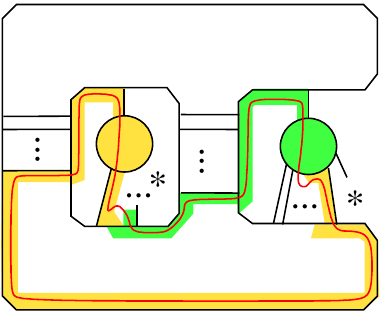} \\
    (a) & (b)
  \end{tabular}
  \caption{Type (3N) and (1S)}
  \label{fig:+-+-3N1S}
\end{figure}

\underline{White edges (3N), (2S).}
In this case, the diagram is as shown in Figure~\ref{fig:+-+-3N2S}(a). Note that (3N) implies that $T_c$ has no state circles, and (2S) implies that $T_d$ has only one state circle. The magenta tentacle originates in $T_b$. The gold face must also originate in $T_b$. This is possible if the gold tentacle either runs across $T_a$, and originates in the northwest of $T_b$, or if the gold tentacle runs across the inside north of the diagram, and originates in the southeast of $T_b$. The two options are shown in Figure~\ref{fig:+-+-3N2S}(b) and~(c). In the first case, we obtain an EPD that may be complex. In the second case, the EPD is semi-simple.

\begin{figure}
  \centering
  \begin{tabular}{ccc}
    \includegraphics{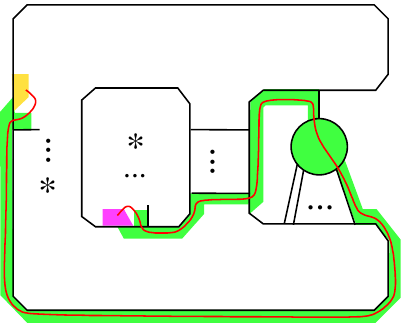} &
    \includegraphics{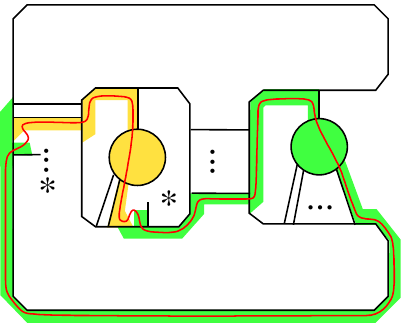} &
    \includegraphics{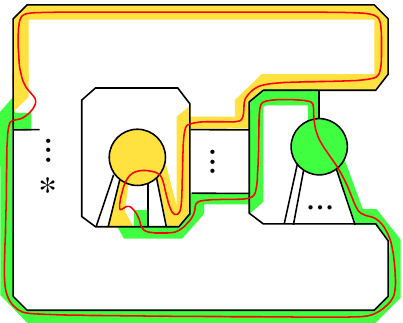} \\
    (a) & (b) & (c)
  \end{tabular}
  \caption{Type (3N) and (2S)}
  \label{fig:+-+-3N2S}
\end{figure}

\underline{White edges (3N), (3S).}
The case of type (3N) and (3S) is shown in Figure~\ref{fig:+-+-3N3S}. The gold faces meeting the green tentacles must both originate in $T_b$, at the southeast corner. Hence the EPD is as shown; it is semi-simple.

\begin{figure}
  \centering
  \includegraphics{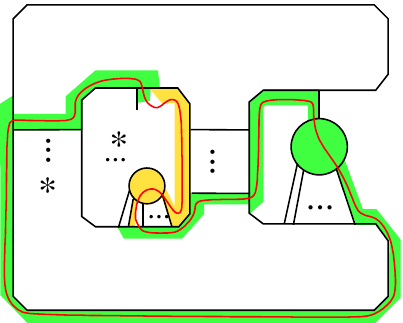}
  \caption{Type (3N) and (3S)}
  \label{fig:+-+-3N3S}
\end{figure}

This concludes the search for type~(1) complex EPDs in $P_A$. The complex EPDs found are in Figures~\ref{fig:+-+-1N1S}(b), ~\ref{fig:+-+-2N1S}(c), and~\ref{fig:+-+-3N2S}(b).
\end{proof}

\begin{lem}
Let $K$ be a $+-+-$ link with reduced, admissible diagram $D(K)$. Then there are no complex EPDs in $P_A$ which are not type~(1).
\label{lem:+-+-type2EPDs}
\end{lem}

\begin{proof}
Let $E$ be a complex EPD in $P_A$ that is not type~(1); then $E$ must be type~(2). That is, $\partial E$ runs along segments across both $T_a$ and $T_c$. For convenience, let's call these two segments $\alpha$ and $\beta$. Now, $\partial E$ may run along either the north or south sides of these segments. Therefore we have four possibilities, listed below. By Lemma~\ref{lemma:pullnormalsquare}, we may pull $E$  into a normal square with white edges at tails of the gold tentacles. 

\underline{Case $\alpha$N, $\beta$N.}
We interpret the notation $\alpha$N, $\beta$N to mean that $\partial E$ runs along the north side of $\alpha$ and the north side of $\beta$. The tentacle running along the north side of $\alpha$ originates in the southeast state circle in $T_d$. The tentacle running along the north side of $\beta$ originates in the southeast state circle in $T_b$. These two cannot agree, hence one is green and one is gold; say the tentacle running along $\alpha$ is green, the one along $\beta$ is gold. 

After running west to east along the north of $\alpha$, $\partial E$ must meet a white face at the head of a green tentacle. The only option is that there is a white face at the north of $T_b$. Similarly, after running west to east along the north of $\beta$, $\partial E$ must meet a white face at the tail of gold. There are three options. First, if there is only one state circle in $T_d$, then the tentacle along the west side of $T_d$ will be green, and $\partial E$ could meet the head of a green tentacle on the east side of $T_c$. This is shown in Figure~\ref{fig:+-+-aNbN}(a); note the result is type~(1), contrary to assumption. Second, again if there is only one state circle in $T_d$, then $\partial E$ could run across the north of $T_d$, into a tentacle in $T_d$, to the head of a green tentacle in $T_d$. This is shown in Figure~\ref{fig:+-+-aNbN}(b). Note this is again type~(1), contradicting assumption. Finally, $\partial E$ could follow a gold tentacle all across the inside north of the diagram, meeting the head of a green tentacle in $T_a$. This is shown in Figure~\ref{fig:+-+-aNbN}(c); note the result is a simple EPD.

\begin{figure}
  \centering
  \begin{tabular}{ccc}
    \includegraphics{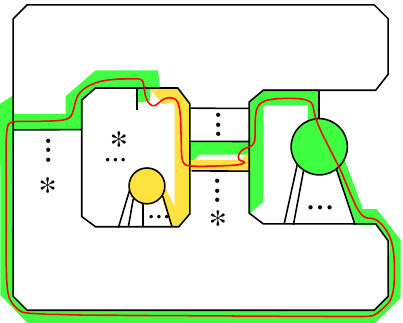} &
    \includegraphics{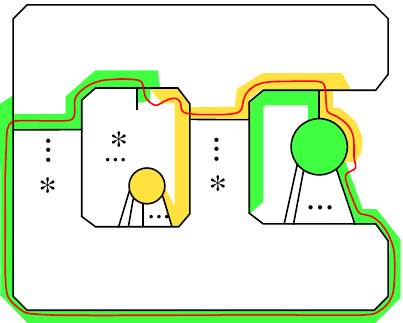} &
    \includegraphics{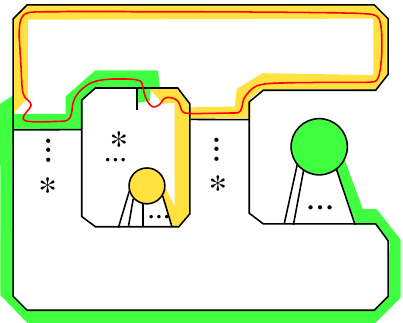} \\
    (a) & (b) & (c)
  \end{tabular}
  \caption{$\partial E$ runs along the north of $\alpha$, north of $\beta$}
  \label{fig:+-+-aNbN}
\end{figure}

\underline{Case $\alpha$N, $\beta$S.}
Again, the tentacle running over $\alpha$ to the north originates in the southeast state circle of $T_d$; color this face green. The tentacle running over $\beta$ to the south originates in the northwest state circle of $T_d$. Suppose first that these state circles of $T_d$ are not the same; so we color the face running over $\beta$ to the south gold. Now, as in the previous case $\partial E$ must run east out of $\alpha$ to a white edge, which must be at a head of green. There is only one possible head of green, at the north of $T_b$, as before. However, at this white edge $\partial E$ runs into the face originating in the southeast of $T_b$; we color this face magenta. Then $\partial E$ runs through the green, magenta, and gold faces, which are all distinct. This contradicts Lemma~\ref{lemma:twoface}.

Therefore, it must be that the northwest and southeast state circles of $T_d$ are the same, and this face is green. Note that $\partial E$ runs from $\beta$ to a white face, at the head of a green tentacle. The only way to meet the head of a green tentacle running west across $\beta$ is if $T_c$ has no state circles, $\beta$ is the far south segment in $T_c$, and $\partial E$ runs around the outside south of $T_b$ to meet a white face just inside $T_c$. Then $\partial E$ must be as shown in Figure~\ref{fig:+-+-aNbS}. This is a type~(1) EPD, contrary to assumption.

\begin{figure}
  \centering
  \includegraphics{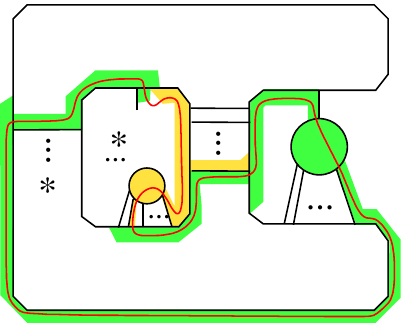}
  \caption{$\partial E$ runs along the north of $\alpha$, south of $\beta$}
  \label{fig:+-+-aNbS}
\end{figure}

\underline{Case $\alpha$S, $\beta$N.}
The tentacle running over the south of $\alpha$ originates in the northwest state circle in $T_b$. The tentacle running over the north of $\beta$ originates in the southeast of $T_b$. If these two state circles are the same, then between $\alpha$ and $\beta$, $\partial E$ must run over $T_b$ from north to south, and $E$ is of type~(1); contradiction.

So assume the northwest and southeast state circles of $T_b$ are distinct. Color the northwest state circle gold, and the southeast one green. After $\partial E$ runs east to west along the gold face to the south of $\alpha$, it must meet a white face at the tail of a gold tentacle. There are three ways this can happen, but we claim that all lead to a face originating in $T_d$, hence give rise to a distinct color. First, the gold tentacle to the south of $\alpha$ can terminate in $T_a$. In this case $\partial E$ meets a white face in $T_a$ and jumps to the tentacle running around the outside south of the diagram. This face originates in the southeast of $T_d$. So suppose, second, the gold tentacle south of $\alpha$ runs all along the inside south of the diagram, and $\partial E$ follows it to where it terminates, in $T_c$. Then it jumps to a green tentacle originating in the northwest of $T_d$. So finally, suppose the gold tentacle south of $\alpha$ runs along the inside south of the diagram, but $\partial E$ follows a new gold tentacle from the south of $T_d$ into $T_d$. Then this tentacle has its tail on a state circle in $T_d$, which must be where the green face originates. All three options require the green face to originate in $T_d$ and in $T_b$, which is impossible. 

\underline{Case $\alpha$S, $\beta$S.}
As in the previous case, the tentacle on the south of $\alpha$ originates in the north of $T_b$. The tentacle on the south of $\beta$ originates in the north of $T_d$. Color the tentacle south of $\alpha$ gold, and the one south of $\beta$ green. Note that $\partial E$, after running east to west along $\beta$, must run to a white edge at the head of a green tentacle. The only way this is possible is if $T_c$ contains no state circles, and $\beta$ is the segment at the far south of $T_c$. The white face must occur at the south of $T_b$, jumping from the head of a green tentacle to the tail of a gold tentacle. This gold tentacle originates in a state circle in $T_b$. This state circle must agree with the origin of the gold tentacle at the south of $\alpha$. In order to form an EPD that is not of type~(1), $\partial E$ must run over the far west side of $T_b$, as in Figure~\ref{fig:+-+-aSbS}(a).

\begin{figure}
\centering
\begin{tabular}{cc}
\includegraphics{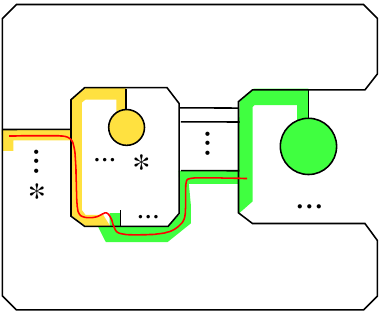} &
\includegraphics{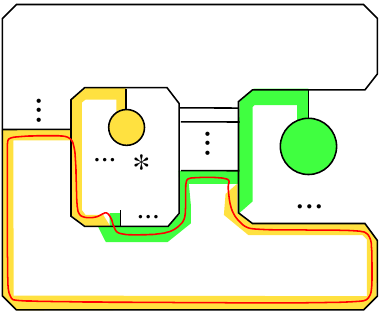} \\
(a) & (b) \\
\end{tabular}
\caption{$\partial E$ runs along the south of $\alpha$, south of $\beta$}
\label{fig:+-+-aSbS}
\end{figure}

Now after $\partial E$ runs east to west along the south of $\alpha$, it must run to the tail of a gold tentacle and the head of a green. Just as in the previous case, there are three possibilities: the gold tentacle terminates in $T_a$, the gold tentacle runs all along the inside south of the diagram and terminates in $T_c$ and $\partial E$ follows it to $T_c$, or the gold tentacle runs along the inside south of the diagram, but $\partial E$ follows a new gold tentacle into $T_d$. The first and last cases can only happen if there is just one state circle in $T_d$ and $\partial E$ runs from south to north of $T_d$, contradicting the fact that $E$ is not type~(1). The middle case is shown in Figure~\ref{fig:+-+-aSbS}(b). Note it results in an EPD bounding a bigon, so it is not complex. 

This completes the analysis of all cases.
\end{proof}

\begin{prop}
  Let $K$ be a $+-+-$ link with reduced, admissible diagram $D(K)$. Then $||E_c|| \leq 1$,  where $||E_c||$ is the number of complex EPDs required to span $P_A$.
  \label{prop:+-+-prop}
\end{prop}

\begin{proof}
Again to obtain the desired result, we need to show that if $E$ and $E'$ are two complex EPDs in $P_A$, then $E'$ is equivalent under parabolic compression to a subset of $E\cup E_s$.

In Lemmas~\ref{lem:+-+-type1EPDs} and~\ref{lem:+-+-type2EPDs}, we found all possible complex EPDs in $P_A$. These are all type~(1), and appear in Figures~\ref{fig:+-+-1N1S}(b), \ref{fig:+-+-2N1S}(c), and~\ref{fig:+-+-3N2S}(b), which we reproduce in Figure~\ref{fig:+-+-complex} for convenience, calling the EPDs $E_1$, $E_2$, and $E_3$. We compare the complex EPDs pairwise. 

\begin{figure}
  \centering
  \begin{tabular}{ccc}
    \includegraphics[width=1.2in]{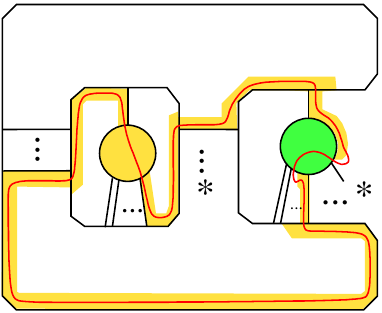} &
    \includegraphics[width=1.2in]{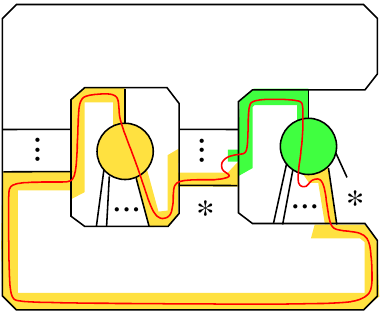} &
    \includegraphics[width=1.2in]{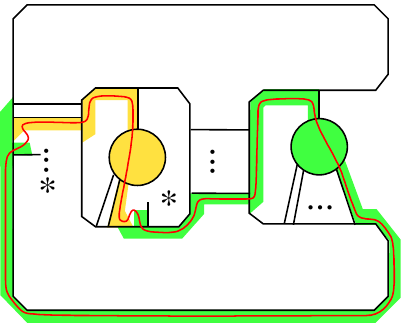} \\
    $E_1$ & $E_2$ & $E_3$
  \end{tabular}
  \caption{Complex EPDs in $P_A$ for a $+-+-$ link}
  \label{fig:+-+-complex}
\end{figure}

Consider first the pair $E_1$ and $E_2$. Both run over the same segments of $T_a$ and $T_b$. In $T_c$, $E_1$ runs over the northernmost segment while $E_2$ runs over a different, parallel segment. However, both then run to the northernmost segment of $T_d$, and then to the south of $T_d$ where they meet again. Thus these two EPDs differ only by a collection of bigons in $T_c$ and possibly a collection of bigons in $T_d$. By Lemma~\ref{lem:equivalentunderparaboliccompression}, the two EPDs are equivalent under parabolic compression.

If $E_1$ and $E_3$ both appear in $P_A$, then note that the asterisks in both figures will be replaced by bigons, and again the two EPDs differ only by bigons. Finally, $E_2$ and $E_3$ also differ only by bigons (in fact, $E_2$ is a cyclic permutation of $E_3$). Lemma~\ref{lem:equivalentunderparaboliccompression} implies the EPDs are equivalent under parabolic compression.
\end{proof}

\section{Proofs of Main Results}\label{sec:proofs}
In this section, we complete the proofs of the main theorems. To do so, we use the following theorem from \cite{FKP}.

\begin{thm}[Theorem~5.14 \cite{FKP}]
Let $D = D(K)$ be a prime $A$--adequate diagram of a hyperbolic link $K$. Then 
\[
\text{vol}(S^3 \setminus K) \geq v_8 (\chi_-(\G'_A) - ||E_c||),
\]
where $\chi_-(\cdot)$ is the negative Euler characteristic, $\G'_A$ is the reduced $A$--state graph, and $\E$ is the number of essential product disks required to span the upper polyhedron of the knot complement.
\label{thm:estimateguts}
\end{thm}

\begin{proof}[Proof of Theorem~\ref{thm:KJmain}]
Let $K$ be a hyperbolic Montesinos link with reduced admissible diagram. 
We will use Theorem~\ref{thm:estimateguts} to bound $vol(S^3 \setminus K)$. The hypotheses of the theorem are that $K$ is prime and that $D$ is $A$--adequate. The diagram $D$ is prime by Proposition~\ref{prop:montesinosPrime}. By Theorem~\ref{thm:lickorish}, the diagram is either $A$ or $B$--adequate. If not $A$--adequate, then it must be $B$--adequate, and so the mirror image is $A$--adequate. Taking the mirror image does not change the volume of the knot complement, so we may assume that the diagram is $A$--adequate.

Now we may use Theorem~\ref{thm:estimateguts}. Notice that we obtain the desired volume bound if we show $||E_c|| \leq 1$. We divide into several simple cases.

\textbf{Case 1.} Suppose $K$ has either all positive or all negative tangles. Then the diagram is alternating, so $||E_c|| = 0$ (see \cite{lackenby}).

\textbf{Case 2.}
Suppose $K$ has three tangles, with slopes not all the same sign. If $K$ is a $++-$ link, then by Proposition~\ref{prop:++-prop}, $||E_c|| \leq 1$, so the desired volume bound holds. If $K$ is a $+--$ tangle, then apply that proposition to the mirror image. These are the only links in this case, up to cyclic permutation.

\textbf{Case 3.} Suppose $K$ has four tangles, with slopes not all the same sign. Up to cyclic permutation, there are four ways the positive and negative tangles can be arranged. These are:

\begin{enumerate}
\item $+-+-$: By Proposition~\ref{prop:+-+-prop}, $||E_c|| \leq 1$. Thus the desired volume bound holds.

\item $++--$: As explained in Remark~\ref{remark:++--}, the link is equivalent to a link of type $+-+-$, and so $||E_c||\leq 1$. 

\item $+++-$ or $+---$: By Theorem~\ref{thm:threepositive}, $||E_c|| = 0$ for such a link or its mirror image.
\end{enumerate}

\textbf{Case 4.} Suppose $K$ has five or more tangles. Then $K$ must have at least three positive or at least three negative tangles. Thus Theorem~\ref{thm:threepositive} applies to $K$ or its mirror image.

Therefore $||E_c|| \leq 1$ in all cases.
\end{proof}

We now turn our attention to the proof of Theorem~\ref{thm:KJ2}. This theorem generalizes \cite[Theorem~9.12]{FKP}, which requires at least three positive and at least three negative tangles. Much of the proof of that result goes through in the setting of fewer positive and negative tangles. 

\begin{lem}\label{lem:twistnumber}
Let $K$ be a Montesinos link that admits a reduced, admissible diagram $D$ with at least two positive and at least two negative tangles. Let $\G'_A$ and $\G'_B$ be the reduced all--$A$ and all--$B$ graphs associated to $D$. Then 
\[
-\chi(\G'_A) - \chi(\G'_B) \geq t(K) - Q_{1/2}(K) - 2
\]
where $Q_{1/2}(K)$ is the number of rational tangles whose slope has absolute value $|q| \in [1/2,1)$.
\end{lem}

\begin{proof}
We follow the proof of \cite[Lemma~9.10]{FKP}. Since $D$ has at least two positive and at least two negative tangles, $D$ is both $A$ and $B$--adequate. Let $v_A$ be the number of vertices in $\G_A$, $e_A$ the number of edges in $\G_A$, and $e'_A$ the number of edges in $\G'_A$; and similarly for $v_B$, $e_B$, and $e'_B$. Then we have $-\chi(\G_A) = v_A - e_A$ and $-\chi(\G'_A) = v_A - e'_A$, and likewise for $-\chi(\G_B)$ and $-\chi(\G'_A)$.

Now construct the Turaev surface $T$ for $D$, as in \cite[Section~4]{dasbach-futer...}. The diagram $D$ will be alternating on $T$, and the graphs $\G_A$ and $\G_B$ embed in $T$ as graphs of the alternating projection, and are dual to one another, and so the number of regions in the complement of $\G_A$ on $T$ is equal to $v_B$. Because $K$ is a cyclic sum of
alternating tangles, $T$ is a torus, just as in \cite[Lemma~9.10]{FKP}. Thus we have
\[
v_A - e_A + v_B = \chi(T) = 0
\]

Now consider the number of edges of $\G_A$ that are discarded when we pass to $\G'_A$. By \cite[Lemma~8.14]{FKP}, edges may be lost in three ways:

\begin{enumerate}
\item If $r$ is a twist region with $c(r)>1$ crossings such that the $A$--resolution of $r$ gives $c(r)$ parallel segments, then $c(r) - 1$ of these edges will be discarded when we pass to $\G'_A$.

\item If $N_i$ is a negative tangle with slope $q \in (-1, -1/2]$, then one edge of $\G_A$ will be lost from the 2--edge loop spanning $N_i$ from north to south. 

\item If there are exactly two positive tangles $P_1$ and $P_2$, then one edge of $\G_A$ will be lost from the 2--edge loop that runs across $P_1$ and $P_2$ from east to west. 
\end{enumerate}

The same holds for $\G_B$, with $A$ and $B$ switched and positive and negative tangles switched. Combining this information, we obtain
\[ \begin{array}{rcl}
(e_A - e'_A) + (e_B - e'_B) & \leq & \sum \{c(r)-1\} +  \# \{ i: |q_i| \in [1/2, 1) \} + 2 \\
 & = & c(D) - t(D) + Q_{1/2}(D) + 2
\end{array} \]
where the sum is over all twist regions $r$.

Since the edges of $\G_B$ are in one--to--one correspondence with the crossings of $D$, we have
\[ \begin{array}{rclll}
-\chi(\G'_A) - \chi(G'_B) & = & e'_A + e'_B - v_A - v_B &   & \\
 & = & (e'_A + e'_B - e_A-e_B )& + e_B &  +  (e_A - v_A - v_B) \\
  & \geq & -c(D) + t(D) - Q_{1/2}(D) - 2 &+ c(D) & +  0 \\
  & = & t(D)  - Q_{1/2}(D) - 2. & & 
\end{array}\qedhere \]
\end{proof}

\begin{lemma}\label{lemma:9.11}
Let $D$ be a reduced, admissible Montesinos diagram with at least two positive tangles and at least two negative tangles. Then
\[ Q_{1/2} \leq \frac{t(D) + \# K}{2},\]
where $\# K$ denotes the number of link components of $K$. 
\end{lemma}

\begin{proof}
Again we follow the proof of \cite[Lemma~9.11]{FKP}. The number $Q_{1/2}(K)$ is equal to the number of tangles with slope $q$ satisfying $|q| \in (1/2, 1)$ plus the number of tangles with slope $q$ satisfying $|q|=1/2$.

A tangle $R_i$ of slope $|q_i|\in (1/2, 1)$ has at least two twist regions, $t(R_i)\geq 2$. Thus $t(R_i)/2 \geq 1$ for such a tangle. 

A tangle of slope $|q|=1/2$ has only one twist region. It can be replaced by a tangle of slope $\infty$ without changing the number of link components of the diagram, but such a replacement gives a diagram with a ``break'' in it. Thus if $n$ is the number of tangles of slope $|q|=1/2$, then $n \leq \# K$. Hence we have
\[ \begin{array}{rcl}
Q_{1/2}(D) &=& \displaystyle{\sum_{ \{R_i \mbox{ tangle}: |q_i|\in (1/2, 1)\}} 1 + \sum_{\{R_j \mbox{ tangle} : |q_j| = 1/2\}} 1} \\[4ex]
  &\leq& \displaystyle{\sum_{\{R_i : |q_i|\in(1/2, 1)} \frac{t(R_i)}{2} + \sum_{\{R_j : |q_j|=1/2\}} \frac{t(R_j) + 1}{2}} \\[4ex]
  &\leq& \displaystyle{\frac{t(D) + \# K}{2}} 
\end{array}
\]\qedhere
\end{proof}

\begin{proof}[Proof of Theorem~\ref{thm:KJ2}]
If $K$ admits a reduced, admissible diagram with at least two positive tangles, then by work of Bonahon and Siebenmann \cite{bonahon2010}, the complement of $K$ must be hyperbolic, unless $K$ is the $(2, -2, 2, -2)$ pretzel link (see also \cite[Section~3.3]{fg:arborescent}). We exclude this pretzel link. 

For the lower bound in the theorem, by Lemmas~\ref{lem:twistnumber} and~\ref{lemma:9.11}, we have
\[
-\chi(\G'_A)-\chi(\G'_B) \geq \frac{t(D) - \# K - 4}{2}.
\]

By Theorem~\ref{thm:essential}, $S_A$ and $S_B$ are both essential in $S^3 \setminus A$. Then a theorem of Agol, Storm, and Thurston \cite[Theorem~9.1]{AST} applied to $S_A$ and $S_B$ implies that 
\[
\vol(S^3 \setminus K) \geq  \frac{v_8}{2}
\left(\chi_-\guts(S^3 \cutalong S_A) + \chi_-\guts(S^3 \cutalong S_B)\right).
\]
By \cite[Theorem~5.14]{FKP}, $\guts(S^3\cutalong S_A) = \chi_-(\G'_A)-||E_c||$; similarly for $\guts(S^3\cutalong S_B)$. By Theorem~\ref{thm:estimateguts}, along with Propositions~\ref{prop:++-prop} and~\ref{prop:+-+-prop}, we may assume that $||E_c||\leq 1$ (possibly after a mutation of the diagram in the $++--$ case), for both the $A$ and $B$ cases. Thus
\[\begin{array}{rcl}
\vol(S^3\setminus K) &\geq&
\displaystyle{\frac{v_8}{2}}\left(\chi_-(\G'_A) + \chi_-(\G'_B)\right) - \displaystyle{\frac{v_8}{2}} \left(||E_c||_A + ||E_c||_B \right) \\[4mm]

& \geq & \displaystyle{\frac{v_8}{4}(t(D) - \# K - 4)- \frac{v_8}{2} (2)} \\
& \geq & \displaystyle{\frac{v_8}{4} (t(D) - \# K - 8)}. 
\end{array}
\]

As for the upper bound on volume, this goes straight through as in the proof of \cite[Theorem~9.12]{FKP} without change. That is, augment the Montesinos link by drilling out a link component $B$ encircling two strands at the east of the tangle. The result is hyperbolic, and a belted sum of tangles as in \cite{adams:3-punct}. The estimates on volume follow. 
\end{proof}

\bibliographystyle{amsplain}

\bibliography{biblio}

\end{document}